\documentclass[11pt]{amsart}
\usepackage{amssymb,latexsym,amsmath,longtable,mathdots}
\usepackage[mathscr]{eucal}
\usepackage{color}

\numberwithin{equation}{section}

\newcommand{\hsp}[1]{{\hbox{\hspace{#1}}}}

\newcommand{\mystack}[2]{\ensuremath{ \substack{ \hbox{\scriptsize{${#1}$}} \\ \hbox{\scriptsize{${#2}$}} }} }

\def\a{\alpha}  
\def\b{\beta}  
\def\c{\gamma}  
\def\d{\delta}  
\def\e{\varepsilon}  

\def\z{\zeta}
\def\m{\mu}
\def\n{\nu}

\def\w{\omega} 

\def\x{\xi}

\def\fa{\mathfrak{a}} \def\sfa{\mathsf{a}}
  
\def\tAd{\mathrm{Ad}} \def\tad{\mathrm{ad}}

\def\cB{\mathcal B}

\def\fb{\mathfrak{b}} 

\def\bC{\mathbb C}

 \def\tdim{\mathrm{dim}}

\def\tEnd{\mathrm{End}}

\def\texp{\mathrm{exp}}
\def\cF{\mathcal F}

  \def\ttF{\mathtt{F}}

\def\tGr{\mathrm{Gr}}
\def\fg{{\mathfrak{g}}}

\def\cH{\mathcal H} 
\def\sfH{\mathsf{H}}

\def\fh{\mathfrak{h}}

\def\tti{\mathtt{i}} 

\def\bI{\mathbf{I}}

\def\tId{\mathrm{Id}}

 \def\ttj{\mathtt{j}} 
  
\def\ttJ{\mathtt{J}}

 \def\ttk{\mathtt{k}}

\def\tLG{\mathrm{LG}}
\def\cM{\mathcal M}

\def\fn{\mathfrak{n}}

 \def\cO{\mathcal O}

\def\bP{\mathbb P}

\def\fp{\mathfrak{p}} 
\def\sfp{\mathsf{p}}

 \def\cQ{\mathcal Q} 
 
\def\fq{\mathfrak{q}} 
\def\sfq{\mathsf{q}} 
 \def\cR{\mathcal R}

\def\sfr{\mathsf{r}}

 \def\cS{\mathcal S}

\def\fs{\mathfrak{s}}
 
\def\tSL{\mathrm{SL}} \def\tSO{\mathrm{SO}}
\def\tSp{\mathrm{Sp}} \def\tSpin{\mathit{Spin}}
 \def\tSing{\mathrm{Sing}}

\def\cY{\mathcal Y} \def\bZ{\mathbb Z}

\def\half{\tfrac{1}{2}}

\def\dfn{\stackrel{\hbox{\tiny{dfn}}}{=}}
\def\sbullet{{\hbox{\tiny{$\bullet$}}}}

\def\inj{\hookrightarrow}

\def\op{\oplus}
\def\ot{\otimes}

\def\tw{\hbox{\small $\bigwedge$}}

\newcounter{numcnt}

\newcounter{cnt}

\newcounter{acnt}
\newenvironment{a_list}{ 
  \begin{list}{{(\alph{acnt})}}
   {\usecounter{acnt} \setlength{\itemsep}{3pt}
    \setlength{\leftmargin}{25pt} \setlength{\labelwidth}{20pt} }
   }
   {\end{list}}

\newcounter{Acnt}

\newcounter{icnt}

\newenvironment{i_list}{ 
  \begin{list}{{(\roman{icnt})}}
   {\usecounter{icnt} \setlength{\itemsep}{3pt}
    \setlength{\leftmargin}{25pt} \setlength{\labelwidth}{20pt} }
   }
   {\end{list}}
\newcounter{Icnt}

\newcounter{exam_cnt}

\newcounter{mccnt}
\newenvironment{circlist}{ 
  \begin{list}{$\circ$}
   {\usecounter{cnt} \setlength{\itemsep}{2pt}
    \setlength{\leftmargin}{15pt} \setlength{\labelwidth}{20pt} }
   }
   {\end{list}}
\newenvironment{bcirclist}{ 
  \begin{list}{\boldmath$\circ$\unboldmath}
   {\usecounter{cnt} \setlength{\itemsep}{2pt}
    \setlength{\leftmargin}{15pt} \setlength{\labelwidth}{20pt} }
   }
   {\end{list}}

\newtheorem{corollary}[equation]{Corollary}
\newtheorem*{corollary*}{Corollary}
\newtheorem{lemma}[equation]{Lemma}
\newtheorem*{lemma*}{Lemma}
\newtheorem{proposition}[equation]{Proposition}
\newtheorem*{proposition*}{Proposition}
\newtheorem{theorem}[equation]{Theorem}
\newtheorem*{theorem*}{Theorem}

\theoremstyle{definition}
\newtheorem*{question*}{Question}          

\newtheorem*{boldQ*}{Question}

\theoremstyle{remark}
\newtheorem*{assume*}{Assume}
\newtheorem*{answer*}{Answer}

\newtheorem*{claim*}{Claim}

\newtheorem{definition}[equation]{Definition}
\newtheorem*{definition*}{Definition}
\newtheorem{example}[equation]{Example}
\newtheorem*{example*}{Example}
\newtheorem*{hint*}{Hint}
\newtheorem*{notation*}{Notation}
\newtheorem{remark}[equation]{Remark}
\newtheorem*{remark*}{Remark}
\newtheorem*{remarks*}{Remarks}
\newtheorem*{fact*}{Fact}
\newtheorem*{emphL*}{Lemma}

\newtheorem*{emphQ*}{Question}

\numberwithin{HWeq}{section}
\theoremstyle{definition}

\begin{document}
\title{Schur flexibility of cominuscule Schubert varieties}
\author[Robles]{C. Robles}
\email{robles@math.tamu.edu}
\address{Mathematics Department, Mail-stop 3368, Texas A\&M University, College Station, TX  77843-3368} 
\thanks{Robles is partially supported by NSF DMS-1006353.}
\date{\today}
\begin{abstract}
Let $X=G/P$ be a cominuscule rational homogeneous variety.  (Equivalently, $X$ admits the structure of a compact Hermitian symmetric space.)  We say a Schubert class $\xi$ is Schur rigid if the only irreducible subvarieties $Y \subset X$ with homology class $[Y] \in \bZ \xi$ are Schubert varieties.  Robles and The identified a sufficient condition for $\xi$ to be Schur rigid.  In this paper we show that the condition is also necessary.
\end{abstract}
\maketitle

\setcounter{tocdepth}{1}

\section{Introduction}

Let $X=G/P$ be a rational homogeneous variety.  The Schubert classes form an additive basis of the integral homology $H_\bullet(X,\bZ)$.    In 1961, Borel and Haefliger \cite{MR0149503} asked: given a Schubert class $\xi_w$ represented by a Schubert variety $Y_w$, aside from the $G$--translates $g\cdot Y_w$, are there any other algebraic representatives of the Schubert class?  In some cases, the Schubert varieties are the only algebraic representatives, and we then say the Schubert class is \emph{rigid}.  This problem has been studied by several people, including R. Hartshorne, E. Rees and E. Thomas \cite{MR0357402}, R. Bryant \cite{SchurRigid}, M. Walters \cite{Walters}, J. Hong \cite{MR2191767, MR2276624}, C. Robles and D. The \cite{MR2960030} and I. Coskun \cite{MR2819546, coskunOrth}.

It is a striking consequence of B. Kostant's work \cite{MR0142696, MR0142697} that, \emph{when $X$ is cominuscule},\footnote{The Grassmannian $\tGr(k,n)$ of complex $k$--planes in $\bC^n$ is an example of a cominuscule rational homogeneous variety (via the Pl\"ucker embedding).} the varieties $Y \subset X$ that are homologous to an integer multiple of the Schubert class $\xi_w$ (that is, $[Y] \in\bZ \xi_w$) are characterized by a system of differential equations known as the \emph{Schur system}, cf. Section \ref{S:DS}.  The Schubert varieties $\{ g \cdot Y_w \ | \ g \in G \}$ are the \emph{trivial solutions}.  When there exist no nontrivial, irreducible solutions, we say that the Schubert class is \emph{Schur rigid}.  Every Schur rigid $\xi_w$ is rigid, and this provides a differential-geometric approach to the algebro-topological question above.

Associated to the differential system is a Lie algebra cohomology.  The cohomology contains a distinguished subspace $\cO_w$.  It is known \cite{MR2960030} that the Schubert class $\xi_w$ is Schur rigid when $\cO_w = 0$; we say that the space $\cO_w$ is a \emph{cohomological obstruction to rigidity}.  The goal of this paper is to show the converse:  when $\cO_w$ is nontrivial, there exist nontrivial solutions $Y$.  Precisely, $Y$ is an irreducible subvariety of $X$ such that (i) $[Y] \in \bZ \xi_w$, and (ii) $Y$ is not a $G$--translate of $Y_w$; this is Corollary \ref{C:SchurFlex}.

From the perspective exterior differential systems (EDS) this construction of non-trivial solutions is very interesting.  It is a relatively common practice to establish rigidity through the machinery of EDS.  In contrast, it is often considerably more difficult to construct non-trivial solutions.  (A striking example is R. Bryant's construction of metrics with exceptional holonomy $G_2$ and $\tSpin(7)$ in \cite{MR916718}.)  Moreover, the construction is typically local.  So, what is particularly interesting here (and I expect will have applications to other problems) is the use of Lie algebra cohomology to construct nontrivial, global (they are algebraic varieties) solutions.

While the construction of the nontrivial solutions $Y$ (in the proof of Theorem \ref{T:SchubertFlex}) is explicit, it is in terms of representation theoretic data, and as a result is not geometrically transparent.  Geometric (and explicit) constructions of $Y$ are given in \cite{CR1}, where a stronger (than Corollary \ref{C:SchurFlex}) statement is proven:  \emph{if $\cO_w \not=0$, then for \emph{any} positive integer $m$ there exists an irreducible variety $Y$ representing $m\xi_w$}.

\subsection{History}
Schur rigidity and the associated differential system were first studied independently by M. Walters \cite{Walters} and R. Bryant \cite{SchurRigid}.  Walters identified Schur rigid classes represented by (a) smooth Schubert varieties in $X = \tGr(k,n)$; and (b) codimension two Schubert varieties in $X = \tGr(2,n)$.  Bryant identified Schur rigid classes represented by (a) smooth Schubert varieties in the cases that $X$ is the Grassmannian $\tGr(k,n)$ or the Lagrangian Grassmannian $\tLG(n,2n)$; (b) the maximal linear subspaces in the classical cominuscule $X$; and (c) singular Schubert varieties of low (co)dimension in $\tGr(k,n)$.    In the case that $Y_w$ is smooth, the results of Bryant and Walters, which were obtained case-by-case, were generalized to arbitrary cominuscule $X$ and given a beautiful, uniform proof by J. Hong \cite{MR2276624}.  There, the obstructions to rigidity are realized as a (Lie algebra) cohomological condition.  Hong also showed that a large class of the singular Schubert varieties in the Grassmannian are Schur rigid \cite{MR2191767}.  C. Robles and D. The \cite{MR2960030} extended the approach of \cite{MR2276624} to give a complete list of the Schubert classes for which there exist no cohomological obstructions to rigidity: these classes are necessarily Schur rigid.  

Several people have worked on the more restrictive problem of determining when $\xi_w$ is rigid, as well as the related problem of determining when $\xi_w$ admits a smooth representative; see \cite{MR2819546, coskunOrth, MR0357402, MR0224611, MR0265371, MR0285535}.

\subsection{Acknowledgements}
Over the course of this project, I have benefitted from conversations and/or correspondence with many people, including R. Bryant, I. Coskun, J.M. Landsberg, N. Ressayre, F. Sottile and D. The.  I thank them for their insights and time.  I am especially grateful to the anonymous referees for several ameliorating suggestions.

\subsection{Contents}
The main result of this paper is Theorem \ref{T:SchubertFlex}.  The theorem yields Corollary \ref{C:SchurFlex}, which is the result discussed above: if $\cO_w \not=0$, then the Schubert class $\xi_w$ is Schur flexible; that is, there exists an irreducible variety $Y$, which is \emph{not} $G$--equivalent to $Y_w$, such that $[Y] \in \bZ\xi_w$.  As a corollary, we find that a Schubert class is Schur rigid if and only if its Poincar\'e dual is Schur rigid (Corollary \ref{C:PD}).  

In Section \ref{S:statements}, we apply Corollary \ref{C:SchurFlex} to each of the classical (irreducible) cominuscule varieties in order to enumerate their Schur flexible classes (Theorems \ref{T:quadrics}, \ref{T:Gr}, \ref{T:LG} and \ref{T:Spinor}).  In the case of the Lagrangian Grassmannian $\tLG(n,2n) = C_n/P_n$ and the Spinor variety $\cS_{n+1} = D_{n+1}/P_{n+1}$ we obtain the following corollary: There is an bijection between the set of Schubert classes in $\tLG(n,2n)$, and the set of Schubert classes in $\cS_{n+1}$.  (This bijection preserves the partial order on the Hasse posets parameterizing the Schubert classes.)  By Corollary \ref{C:LGspin}, a Schubert class in $\tLG(n,2n)$ is Schur rigid if and only if the corresponding class in $\cS_{n+1}$ is Schur rigid.

The Schur rigid classes of the two exceptional (irreducible) cominuscule varieties are given in Figures \ref{F:E6} and \ref{F:E7}.

\tableofcontents\listoffigures\listoftables

\section{Rigid classes in the irreducible {$G/P$}}
\label{S:statements}

\noindent The irreducible, cominuscule rational homogeneous varieties $X = G/P$ are:
\begin{bcirclist}
\item the Grassmannian $\tGr(k,\bC^n) \simeq \tSL_n(\bC)/P_k$ of $k$--planes in $\bC^n$;
\item the smooth quadric hypersurface $\cQ^m \subset \bP^{m+1} \simeq \tSO(m+2)/P_1$;
\item the Lagrangian Grassmannian $\tLG(n,\bC^{2n}) = \tSp_{2n}(\bC)/P_n$;
\item the Spinor variety $\cS \simeq \tSpin_{2n}(\bC)/P_n$;
\item the Cayley plane $E_6/P_6$, and the Freudenthal variety $E_7/P_7$.
\end{bcirclist}
The Grassmannians, quadrics, Lagrangian Grassmannians and Spinors, are the \emph{classical} $X$; the Cayley plane and Freudenthal variety are the \emph{exceptional} $X$.  

Theorems \ref{T:Gr}--\ref{T:Spinor} list the Schur rigid classes for each of the classical $X$ above.  These theorems are proved as follows.  By Corollary \ref{C:SchurFlex}, a Schubert class $\xi_w \in H_\sbullet(X,\bZ)$ is Schur rigid if and only if $\cO_w=0$; equivalently, the condition $\sfH_+$ (Definition \ref{D:H1H2}) is satisfied.  For each of the irreducible, cominuscule $X$, the Schubert classes satisfying the condition $\sfH_+$ are given by \cite[Theorem 6.1]{MR2960030}.  There, the Schubert classes are described in terms of representation theoretic data, cf. Section \ref{S:aJ}.  This representation theoretic description is well-suited to the proof of Theorem \ref{T:SchubertFlex}, but is not geometrically transparent.  The appendices give the translations between the representation theoretic descriptions and the familiar geometric descriptions for the classical $X$; see, in particular, Lemmas \ref{L:part}, \ref{L:Cpart} and \ref{L:Dpart}.  These translations, applied to Corollary \ref{C:SchurFlex} and \cite[Theorem 6.1]{MR2960030}, yield the geometric descriptions of the Schur rigid classes given by Theorems \ref{T:Gr}--\ref{T:Spinor}.    Throughout, 
$$
  o \ = \ P/P \ \in \ G/P \ = \ X \,.
$$

In the case of the two exceptional $X$, the Schubert classes satisfying the condition $\sfH_+$ (equivalently, by Corollary \ref{C:SchurFlex}, the Schur rigid classes) are given by \cite[Tables 4 and 5]{MR2960030}.  This data is transcribed to Figures \ref{F:E6} and \ref{F:E7} (pages \pageref{F:E6} and \pageref{F:E7}).  

\subsection*{Key to Figures \ref{F:E6} and \ref{F:E7}} \label{S:key} 
These figures are respectively the Hasse posets of the Cayley plane $E_6/P_6$ and Freudenthal variety $E_7/P_7$.  Each node represents a Schubert class $\xi_w$ and is labeled with degree of $Y_w$ in the minimal homogeneous embeddings $E_6/P_6 \inj \bP V_{\w_6}$ and $E_7/P_7 \inj \bP V_{\w_7}$, cf. \cite[Section 4.8]{MR1782635}.  (In fact, the Cayley plane and Freudenthal variety are also minuscule.  So \cite[Remark 4.8.4]{MR1782635} may be used to compute the degree.)  The height of the node indicates the dimension of $\xi_w$; in particular, the lowest node $o \in X$ is at height zero.  Two nodes are connected if the Schubert variety associated with the lower node is a divisor of the Schubert variety associated with the higher node.  The node is circled if the corresponding Schubert class is Schur rigid.  

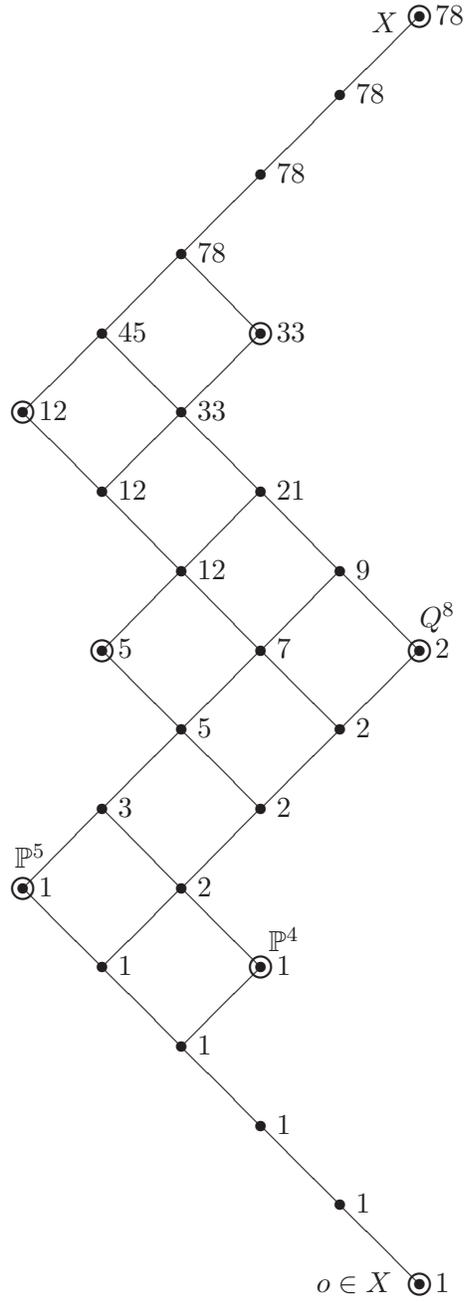
\begin{figure}[p]
\caption[Hasse poset of $E_6/P_6$]{Hasse poset of the Cayley plane $E_6/P_6$, cf. page \pageref{S:key}.}
\setlength{\unitlength}{30pt}
\begin{picture}(5,17)
\put(5,0){\line(-1,1){5}}\put(3,4){\line(-1,1){2}}
\put(3,6){\line(-1,1){2}}\put(4,7){\line(-1,1){4}}
\put(5,8){\line(-1,1){4}}\put(3,12){\line(-1,1){1}}
\put(0,11){\line(1,1){5}}\put(1,10){\line(1,1){2}}
\put(1,8){\line(1,1){2}}\put(0,5){\line(1,1){4}}
\put(1,4){\line(1,1){4}}\put(2,3){\line(1,1){1}}
\multiput(5,0)(-1,1){6}{\circle*{0.15}}
\multiput(3,4)(-1,1){3}{\circle*{0.15}}
\multiput(3,6)(-1,1){3}{\circle*{0.15}}
\multiput(4,7)(-1,1){5}{\circle*{0.15}}
\multiput(5,8)(-1,1){5}{\circle*{0.15}}
\multiput(3,12)(-1,1){2}{\circle*{0.15}}
\multiput(5,16)(-1,-1){3}{\circle*{0.15}}
\thicklines
\put(3,4){\circle{0.25}} \put(0,5){\circle{0.25}}
\put(1,8){\circle{0.25}} \put(5,8){\circle{0.25}}
\put(0,11){\circle{0.25}} \put(3,12){\circle{0.25}}
\put(5,0){\circle{0.25}} \put(5,16){\circle{0.256}}
\thinlines
\put(4.2,14.9){78} \put(3.2,13.9){78} \put(2.2,12.9){78}
\put(3.2,11.9){33} \put(1.2,11.9){45}
\put(0.2,10.9){12} 
\put(2.2,10.9){33}\put(1.2,9.9){12} 
\put(3.2,9.9){21} \put(2.2,8.9){12} \put(4.2,8.9){9}
\put(1.2,7.9){5} \put(3.2,7.9){7} 
\put(5.2,7.9){2} \put(5.0,8.3){$Q^8$} 
\put(2.2,6.9){5} \put(4.2,6.9){2}
\put(1.2,5.9){3} \put(3.2,5.9){2}
\put(0.2,4.9){1} \put(-0.1,5.25){$\bP^5$} 
\put(2.2,4.9){2}
\put(1.2,3.9){1} \put(3.2,3.9){1} \put(3.1,4.2){$\bP^4$} 
\put(2.2,2.9){1} \put(3.2,1.9){1} \put(4.2,0.9){1}
\put(5.2,-0.1){1} \put(5.2,15.9){78}
\normalsize
\put(3.7,-0.1){$o \in X$} \put(4.4,15.8){$X$}
\end{picture}
\label{F:E6}
\end{figure}

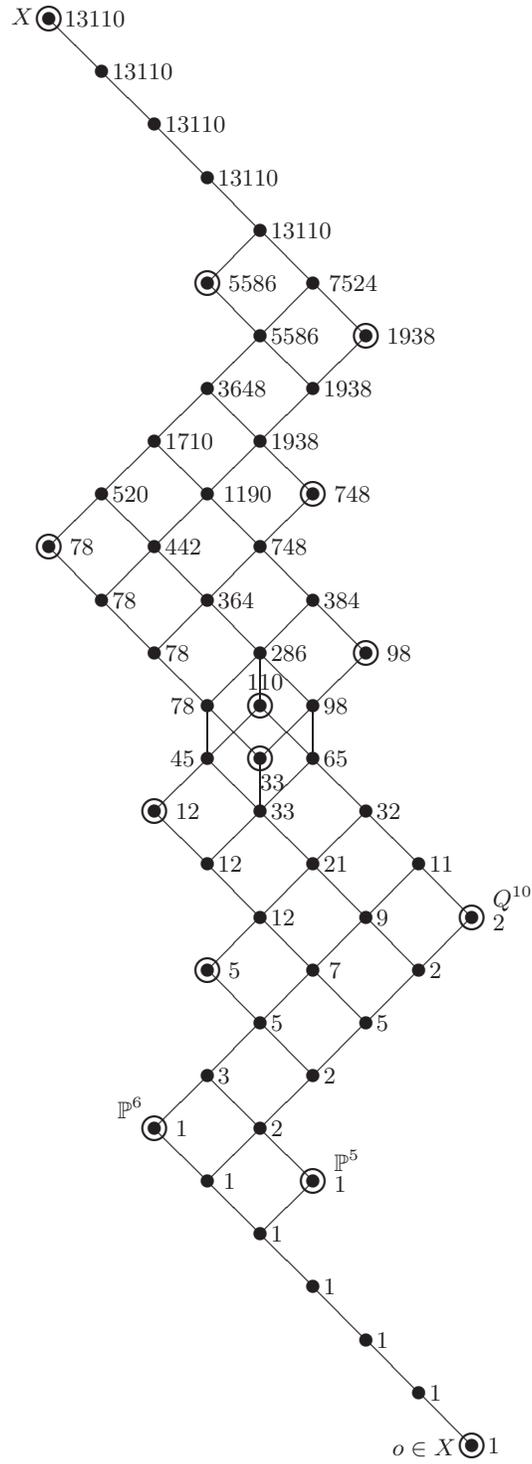
\begin{figure}[p]
\caption[Hasse poset of $E_7/P_7$]{Hasse poset of the Freudenthal variety $E_7/P_7$, cf. page \pageref{S:key}.}
\setlength{\unitlength}{20pt}
\begin{picture}(8,28)
\put(2,6){\line(1,-1){6}} \put(3,7){\line(1,-1){2}}
\put(3,9){\line(1,-1){2}} \put(2,12){\line(1,-1){4}}
\put(3,13){\line(1,-1){4}} \put(4,14){\line(1,-1){4}}
\put(0,17){\line(1,-1){4}} \put(1,18){\line(1,-1){4}}
\put(2,19){\line(1,-1){4}} \put(3,20){\line(1,-1){2}}
\put(3,22){\line(1,-1){2}} \put(0,27){\line(1,-1){6}}
\put(4,4){\line(1,1){1}} \put(3,5){\line(1,1){5}}
\put(2,6){\line(1,1){5}} \put(3,9){\line(1,1){3}}
\put(3,11){\line(1,1){2}} \put(2,12){\line(1,1){2}}
\put(4,13){\line(1,1){2}} \put(3,14){\line(1,1){2}}
\put(4,12){\line(0,1){1}} \put(4,14){\line(0,1){1}}
\put(3,13){\line(0,1){1}} \put(5,13){\line(0,1){1}}
\put(2,15){\line(1,1){3}} \put(1,16){\line(1,1){5}}
\put(0,17){\line(1,1){5}} \put(3,22){\line(1,1){1}}
\multiput(2,6)(1,-1){7}{\circle*{0.25}} \multiput(3,7)(1,-1){3}{\circle*{0.25}}
\multiput(3,9)(1,-1){3}{\circle*{0.25}} \multiput(2,12)(1,-1){5}{\circle*{0.25}}
\multiput(3,13)(1,-1){5}{\circle*{0.25}} \multiput(4,14)(1,-1){5}{\circle*{0.25}}
\multiput(0,17)(1,-1){5}{\circle*{0.25}} \multiput(1,18)(1,-1){5}{\circle*{0.25}}
\multiput(2,19)(1,-1){5}{\circle*{0.25}} \multiput(3,20)(1,-1){3}{\circle*{0.25}}
\multiput(3,22)(1,-1){3}{\circle*{0.25}} \multiput(0,27)(1,-1){7}{\circle*{0.25}}
\thicklines
\put(8,0){\circle{0.45}}
\put(5,5){\circle{0.45}} \put(2,6){\circle{0.45}} 
\put(3,9){\circle{0.45}} \put(8,10){\circle{0.45}} 
\put(2,12){\circle{0.45}} \put(4,13){\circle{0.45}} 
\put(4,14){\circle{0.45}} \put(6,15){\circle{0.45}} 
\put(0,17){\circle{0.45}} \put(6,21){\circle{0.45}} 
\put(3,22){\circle{0.45}} \put(5,18){\circle{0.45}} 
\put(0,27){\circle{0.45}}
\thinlines
\footnotesize
\put(6.5,-0.2){$o \in X$} \put(-0.7,26.9){$X$}
\put(5.4,5.2){$\bP^5$} \put(1.3,6.2){$\bP^6$}
\put(8.4,10.2){$Q^{10}$}
\put(8.3,-0.15){1}
\put(7.2,0.85){1} \put(6.2,1.85){1} \put(5.2,2.85){1}
\put(4.2,3.85){1} \put(5.4,4.8){1} \put(3.3,4.85){1}
\put(4.2,5.85){2} \put(2.4,5.85){1}
\put(5.2,6.85){2} \put(3.2,6.85){3}
\put(4.2,7.85){5} \put(6.2,7.85){5}
\put(3.4,8.85){5} \put(5.3,8.85){7} \put(7.2,8.85){2}
\put(4.2,9.85){12} \put(6.2,9.85){9} \put(8.4,9.75){2}
\put(3.2,10.85){12} \put(5.2,10.85){21} \put(7.2,10.85){11}
\put(2.4,11.85){12} \put(4.2,11.85){33} \put(6.2,11.85){32}
\put(2.3,12.85){45} \put(4.0,12.4){33} \put(5.2,12.85){65}
\put(2.3,13.85){78} \put(3.75,14.3){110} \put(5.2,13.85){98}
\put(2.2,14.85){78} \put(4.2,14.85){286} \put(6.4,14.85){98}
\put(1.2,15.85){78} \put(3.2,15.85){364} \put(5.2,15.85){384}
\put(0.4,16.85){78} 
\put(2.2,16.85){442} \put(4.2,16.85){748}
\put(1.2,17.85){520} \put(3.3,17.85){1190} \put(5.4,17.85){748}
\put(2.2,18.85){1710} \put(4.2,18.85){1938} 
\put(3.2,19.85){3648} \put(5.2,19.85){1938}
\put(4.2,20.85){5586} \put(6.4,20.85){1938}
\put(3.4,21.85){5586} \put(5.3,21.85){7524}
\put(4.2,22.85){13110} \put(3.2,23.85){13110} \put(2.2,24.85){13110}
\put(1.2,25.85){13110} \put(0.3,26.85){13110}
\normalsize
\end{picture}
\label{F:E7}
\end{figure}

\subsection{Quadric hypersurfaces}
\label{S:sQ}

If the quadric hypersurface $\cQ \subset \bP^{2m}$ is of odd dimension, then $[o]$ and $[\cQ]$ are the only Schur rigid Schubert classes.  This is seen by a simple topological argument, as follows.  It is well-known that $H_{2d}(\cQ,\bZ) = \bZ$ for all $0 \le d \le 2m-1$; that is, $\xi_w$ spans $H_{2|w|}(\cQ,\bZ)$ where, $|w| = \tdim_\bC \xi_w$.  (See \cite[p.139--140]{MR1782635} for an explicit description of the Schubert varieties; they are linear subspaces and degenerate quadrics.)  Therefore, if $Y \subset \cQ$ is any subvariety of dimension $d$ (for example, a general complete intersection), then $[Y] \in \bZ_{>0} \xi_w$, where $\xi_w$ is the unique Schubert class of dimension $d$.  Thus, $\xi_w$ is Schur flexible, for all $\tdim_\bC\,\xi_w \not=0,2m-1$.  

If the quadric $\cQ \subset \bP^{2m+1}$ is of even dimension, the again $H_{2d}(\cQ,\bZ) = \bZ$ for all $0 \le d \le 2m$, except $d = m$.  (See \cite[pp. 142--143]{MR1782635} for an explicit description of the Schubert varieties; again, they are linear subspaces and degenerate quadrics.)  As above, the Schubert class $\x_w$ is Schur flexible for all $\tdim_\bC \xi_w \not= 0,m,2m$.  If $d = m$, then $H_{2m}(\cQ,\bZ) = \bZ \op \bZ$, and the two spanning Schubert classes $\xi_w$ and $\xi_{w'}$ are each represented by a maximal linear space.  These two Schubert classes are known to be Schur rigid by the work of Hong.

In summary, we have the following.

\begin{theorem}[Schur rigidity in quadrics] \label{T:quadrics}
\begin{a_list}
\item If the quadric $\cQ\subset\bP^{2m}$ is of odd dimension, then the only Schur rigid Schubert classes are $[o]$ and $[X]$.
\item If the quadric $\cQ\subset\bP^{2m+1}$ is of even dimension, then the Schur rigid Schubert classes are $[o]$, $[X]$ and the two distinct classes $\xi_w, \xi_{w'} \in H_{2m}(\cQ,\bZ)$ represented by maximal linear subspaces.
\end{a_list}
\end{theorem}

\subsection{Grassmannians}
\label{S:sGr}

It is well-known that Schubert varieties in $\tGr(\tti, n+1)$ are indexed by partitions 
\begin{equation} \nonumber 
  \lambda \,=\, (\lambda_1,\ldots,\lambda_\tti) \in \bZ^\tti 
  \quad\hbox{such that}\quad
  1 \le \lambda_1 < \lambda_2 < \cdots < \lambda_\tti \le n+1 \,,
\end{equation}
cf. \cite[\S3.1.3]{MR1782635}.  Fix a flag $0 \subset F^1 \subset F^2 \subset\cdots\subset F^{n+1} = \bC^{n+1}$.  The corresponding Schubert variety is
\begin{equation} \nonumber 
  Y_\lambda(F^\sbullet) \ = \ \{ E \in \tGr(\tti,n+1) \ | \ 
  \tdim( E \cap F^{\lambda_k}) \ge k \,,\ 1 \le k \le \tti \} \,.
\end{equation}
Define $\tilde\lambda_k = \lambda_k - k$; then $0 \le \tilde\lambda_1 \le \cdots \le \lambda_\tti \le n+1-\tti$.  Condense $\tilde\lambda$ by writing $\tilde\lambda = (\n_\sfp^{c_\sfp},\ldots,\n_1^{c_1},\n_0^{c_0})$ with $0 \le \n_\sfp < \cdots < \n_1 < \n_0 \le n+1-\tti$, and $0< c_s$, for all $0 \le s \le \sfp$.  Corollary \ref{C:SchurFlex}, \cite[Theorem 6.1]{MR2960030}, Remark \ref{R:tlambda} and Lemma \ref{L:part} yield the following.

\begin{theorem}[Schur rigidity in Grassmannians] \label{T:Gr}
The Schubert class $\xi_\lambda \in H_\sbullet(\tGr(\tti,n+1),\bZ)$ is Schur rigid if and only if the following conditions hold:
\begin{a_list}
\item $\n_{s-1} - \n_s \ge 2$, for all $1 \le s \le \sfp$;
\item $c_s \ge 2$, for all $1 \le s \le \sfp-1$;
\item if $\n_\sfp > 0$, then $c_\sfp \ge 2$;
\item if $\n_0 < n+1-\tti$, then $c_0 \ge 2$.
\end{a_list}
\end{theorem}

\subsection{Lagrangian Grassmannians}
\label{S:sLG}

It is well-known that Schubert varieties in $\tLG(n,2n)$ are indexed by partitions $1 \le \lambda_1 < \lambda_2 < \cdots < \lambda_n \le 2n$ with the property that $\lambda_i \in \lambda$ if and only if $2n+1-\lambda_i\not\in\lambda$, cf. \cite[\S9.3]{MR1782635}.  Given a non-degenerate skew-symmetric bilinear form on $\bC^{2n}$, fix an isotropic flag $F^\sbullet$ in $\bC^{2n}$.  The corresponding Schubert variety is 
$$
  Y_\lambda(F^\sbullet) \ = \ \{ E \in \tLG(n,2n) \ | \ \tdim(E\cap F^{\lambda_k}) \ge k
  \,,\ 1 \le k \le n \} \, .
$$ 
Define $\tilde\lambda = (\n_\sfp^{c_\sfp},\ldots,\n_0^{c_0})$ as in Section \ref{S:sGr}.  Corollary \ref{C:SchurFlex}, \cite[Theorem 6.1]{MR2960030}, Remark \ref{R:tlambda} and Lemma \ref{L:Cpart} yield the following.

\begin{theorem}[Schur rigidity in Lagrangian Grassmannians] \label{T:LG}
The Schubert class $\xi_\lambda \in H_\sbullet(\tLG(n,2n),\bZ)$ is Schur rigid if and only if the following conditions hold:
\begin{a_list}
\item if $\n_\sfp >0$, then $2 \le c_s$ for all $1 \le s \le \sfp$;
\item if $\n_\sfp = 0$, then $2 \le c_s$ for all $0 \le s \le \sfp-1$.
\end{a_list}
\end{theorem}

\noindent See Table \ref{t:C5P5} (page \pageref{t:C5P5}) for a list of Schur rigid Schubert classes in $\tLG(5,10)$.

\subsection{Spinor varieties}
\label{S:sSpinor}

It is well-known that the Schubert varieties of $\cS_n = D_n/P_n$ are indexed by partitions $1 \le \lambda_1 < \lambda_2 < \cdots < \lambda_n \le 2n$ with the properties that $\lambda_i \in \lambda$ if and only if $2n+1-\lambda_i\not\in\lambda$, and $\# \left\{ i \ | \ \lambda_i > n \right\} $ is even, cf. \cite[\S9.3]{MR1782635}.  Given a non-degenerate symmetric bilinear form on $\bC^{2n}$, fix an isotropic flag $F^\sbullet$ in $\bC^{2n}$.  The corresponding Schubert variety is given by 
\begin{equation} \nonumber 
  Y_\lambda(F^\sbullet) \ = \ \{ E \in \cS_n \ | \ \tdim(E\cap F^{\lambda_k}) \ge k
  \,,\ 1 \le k \le n \} \, .
\end{equation}

Decompose $\lambda = \hat\m_\sfp\cdots\hat\m_1\hat\m_0$ into blocks $\hat\m_s$ of `consecutive' integers, with the convention that the integers $n-1,n+1$ are considered consecutive and are placed in the same $\hat\m_s$--block; likewise, the integers $n,n+2$ are consecutive.  For example, if $n=5$, then $\lambda = (2,3,4,6,10)$ has block decomposition $\hat\m_1\hat\m_0 = (2,3,4,6)(10)$; likewise, $\lambda = (1,2,5,7,8)$ has block decomposition $\hat\m_1\hat\m_0 = (1,2)(5,7,8)$.  

Let $\hat c_s = |\hat\m_s|$ be the number of terms in the $s$--th block, and define
\begin{equation} \nonumber  
  \sfr \,=\, \sfr(\lambda) \ = \ 
  \left\{ \renewcommand{\arraystretch}{1.2} \begin{array}{ll}
    \lfloor \sfp/2 \rfloor & \hbox{ if } \lambda_1 = 1 \,,\\
    \lceil \sfp/2 \rceil & \hbox{ if } \lambda_1 > 1 \,.
  \end{array} \right.
\end{equation}
Then $\hat\m_\sfr$ is the block that has nonempty intersection with $\{ n,n+1\}$.  Corollary \ref{C:SchurFlex}, \cite[Theorem 6.1]{MR2960030} and Lemma \ref{L:Dpart} yield the following.

\begin{theorem}[Schur rigidity in Spinor varieties] \label{T:Spinor}
The Schubert class $\xi_\lambda \in H_\sbullet(\cS_n,\bZ)$ is Schur rigid if and only if one of the following conditions hold:
\begin{a_list}
\item $\lambda_1 > 1$, $2\le \hat c_s$ for all $1\le s\le\sfp$, and $3 \le \hat c_\sfr$;
\item $\lambda_1 = 1$, $2\le \hat c_s$ for all $0 \le s \le \sfp-1$, and $3 \le \hat c_\sfr$.
\end{a_list}
\end{theorem}

\noindent See Table \ref{t:D6P6} (page \pageref{t:D6P6}) for a list of Schur rigid Schubert classes in $\cS_6$.

There exists an inclusion--preserving bijection between the set $\{ Y_{\lambda'} \subset \tLG(n,2n)\}$ of Schubert varieties in the Lagrangian Grassmannian and the set $\{ Y_\lambda \subset \cS_{n+1}\}$ of Schubert varieties in the Spinor variety.  Given a partition $\lambda = (\lambda_1,\ldots,\lambda_{n+1})$ indexing a Schubert variety in $\cS_{n+1}$, the corresponding Schubert variety in the Lagrangian Grassmannian is indexed by $\lambda' = (\lambda'_1,\ldots,\lambda'_n)$, where
$$
  \lambda' \ = \ \left\{
  \begin{array}{ll}
    \lambda'_j = \lambda_j & \hbox{ if } \lambda_j < n+1 \,,\\
    \lambda'_j = \lambda_j - 2 & \hbox{ if } \lambda_j > n+2 \,.
  \end{array} \right.
$$
Basically, $\lambda'$ is obtained from $\lambda$ as follows:  $\lambda$ contains precisely one of $\{n+1,n+2\}$; remove that integer, and shift those above down by two.  

Theorems \ref{T:LG} and \ref{T:Spinor} yield the following.

\begin{corollary} \label{C:LGspin}
The Schubert class $\xi_{\lambda'} \in H_\sbullet(\tLG(n,2n),\bZ)$ is Schur rigid if and only if the corresponding Schubert class $\xi_\lambda \in H_\sbullet(\cS_{n+1},\bZ)$ is Schur rigid.
\end{corollary}

\noindent The corollary is illustrated by a comparison of Tables \ref{t:C5P5} and \ref{t:D6P6}.

\section{Review} \label{S:review}

\subsection{Notation and background} \label{S:not}

This is a continuation of \cite{MR2960030}.  With the exception noted in Remark \ref{R:newaJ}, I will use the notation of that paper.  To streamline the presentation, I will assume that the reader has reviewed the discussion of rational homogeneous varieties, their Schubert subvarieties, grading elements and Hasse posets in \cite[Sections 2.1-2.4 and 3.1]{MR2960030}.  Briefly, $G$ is a complex simple Lie group.  A choice of Cartan and Borel subgroups $H \subset B$ has been fixed, $P \supset B$ is a maximal parabolic subgroup associated with a cominuscule root, and $X = G/P$ is the corresponding cominuscule variety.  The associated Lie algebras are denoted $\fh \subset \fb \subset \fp \subset \fg$.  Let $W$ denote the Weyl group of $\fg$, and $W_\fp$ the Weyl group of the reductive component in the Levi decomposition of $\fp$.  The Hasse poset $W^\fp$ is the set of minimal length representatives of the coset space $W_\fp \backslash W$, and indexes the Schubert classes.  Let 
$$
  o = P/P \ \in \ X = G/P \, .
$$
Given $w \in W^\fp$, the Zariski closure
$$
  Y_w \ = \ \overline{B w^{-1} \cdot o}
$$
is a Schubert variety.  Any $G$--translate of the Schubert variety $Y_w$ will be referred to as a \emph{Schubert variety of type $w$}.  Let $\xi_w = [Y_w] \in H_{2|w|}(X,\bZ)$ denote the corresponding Schubert class. 

Let $\{ Z_1,\ldots,Z_r\}$ be the basis of $\fh$ dual to the simple roots $\{\a_1,\ldots,\a_r\}$.  Let $\a_\tti$ be the simple root associated with the cominuscule $\fp$.  Then $\fg$ decomposes into $Z_\tti$--eigenspaces
\begin{equation} \label{E:gi}
  \fg \ = \ \fg_1 \,\op\, \fg_0 \,\op\, \fg_{-1} \,,\quad\hbox{where}\quad
  \fg_k \, := \, \{ A \in \fg \ | \ [Z_\tti \,,\, A] = k A \} \,;
\end{equation}
we call this eigenspace decomposition the \emph{$Z_\tti$--graded decomposition} of the Lie algebra $\fg$.  Moreover,
\begin{equation} \label{E:p}
  \fp \ = \ \fg_1 \,\op\, \fg_0 \,,
\end{equation}
and $\fg_0$ is the reductive component of the parabolic subalgebra $\fp$.   

\begin{remark} \label{R:abelian}
By the Jacobi identity, $[\fg_k \,,\, \fg_\ell ] \subset \fg_{k+\ell}$.  In particular, the subalgebras $\fg_{\pm1} \subset \fg$ are abelian, $[ \fg_1 \,, \, \fg_1] \ = \ \{0\} \ = \ [ \fg_{-1} \,,\, \fg_{-1} ]$.  
\end{remark}

Let $\Delta$ denote the set of roots of $\fg$.  Given $\a \in \Delta$, let $\fg_\a \subset \fg$ denote the corresponding root space.  Given any subset $\fs \subset \fg$, let 
$$
  \Delta(\fs) \ = \ \{ \a \in \Delta \ | \ \fg_\a \subset \fs \} \,.
$$
Given a subset $U$ of a vector space, let $\langle U \rangle$ denote the linear span.

\subsection{The Schur differential system} \label{S:DS}

There is a natural $\fg_0$--module identification 
$$
  T_oX \ \simeq \ \fg_{-1} \,.
$$  
Kostant \cite[Corollary 8.2]{MR0142696} showed that the $\ell$--th exterior power decomposes as
$$
  \tw^\ell T_oX \ = \ \bigoplus_{\mystack{w\in W^\fp}{|w|=\ell}} \bI_w
$$
into irreducible $\fg_0$--modules. The highest weight line $\fn_w \in \bP \bI_w$ is defined by 
\begin{equation} \label{E:nw}
  \Delta(w) \ = \ w\Delta^-\cap\Delta^+
  \quad\hbox{and}\quad 
  \fn_w \ = \ \bigoplus_{\a\in\Delta(w)} \, \fg_{-\a}
\end{equation}

\begin{remark} \label{R:Dw}
One important consequence of Remark \ref{R:abelian} is that given a set $\Phi \subset \Delta(\fg_1)$, there exists $w \in W^\fp$ such that $\Delta(w) = \Phi$ if and only if $\Delta^+ \backslash \Phi$ is closed.  For details see \cite[Section 2.3]{MR2960030}.
\end{remark}

Consider the set of tangent $\ell$--planes $\tGr(\ell,T_oX)$ as a subvariety of $\bP(\tw^\ell T_oX)$ via the Pl\"ucker embedding, and set $R_w = \tGr(|w|,T_oX) \,\cap\, \bP\,\bI_w$.  The \emph{Schur system} is the homogeneous bundle $\cR_w \subset \tGr(|w|,TX)$ with fibre $R_w$ over $o \in X$.  Given a complex submanifold $M \subset X$ of dimension $|w|$, let $\cM \subset \tGr(|w|,TX)$ denote the canonical lift.  An \emph{integral manifold $M$ of the Schur system} is a complex submanifold $M \subset X$ with the property that $\cM \subset \cR_w$.  The \emph{Schur system is rigid} if every connected integral manifold is contained in a Schubert variety of type $w$.  An integral manifold that is not contained in any Schubert variety of type $w$ we call \emph{nontrivial}.  When there exist nontrivial integrals we say the \emph{Schur system is flexible}.

An \emph{integral variety of the Schur system} is a subvariety $Y \subset X$ with the property that the set of smooth points $Y^0$ is an integral manifold of the Schur system.    The Schubert class $\xi_w$ (or the Schubert variety $Y_w$) is \emph{Schur rigid} if every irreducible integral variety of the Schur system is a Schubert variety of type $w$.  An irreducible integral variety that is not Schubert variety of type $w$ is \emph{nontrivial}.  When there exist nontrivial integral varieties we say the Schubert class $\xi_w$ (or the Schubert variety $Y_w$) is \emph{Schur flexible}.  

\begin{theorem}[{\cite{SchurRigid, Walters}}] \label{T:Schur}
The homology class $[Y] \subset H_{2|w|}(X,\bZ)$ represented by a variety $Y \subset X$ is an integer multiple of the Schubert class $\xi_w$ if and only if $Y$ is an integral of the Schur system $\cR_w$.
\end{theorem}

\begin{remark} \label{R:Schur} 
\emph{A priori}, it may happen that $\cR_w$ admits nontrivial integral manifolds (a differential-geometric property), but no nontrivial integral varieties (an algebraic-geometric property).  That is, it may be the case that $\cR_w$ is flexible, while $\xi_w$ is Schur rigid.  However, we will see \emph{a posteriori} (Corollary \ref{C:SchurFlex}) that $\cR_w$ is flexible if and only if $\xi_w$ is Schur flexible.
\end{remark}

For more on the Schur system see \cite[Section 7]{MR2960030} and the references therein.

\subsection{The Schubert differential system} \label{S:schubert}

Recall the $Z_\tti$--graded decomposition \eqref{E:gi} of $\fg$, and let 
\begin{equation} \nonumber
  G_0 \ := \ \{ g \in G \ | \ \tAd_g(\fg_j) \subset \fg_j\} \, .
\end{equation} 
Then $G_0$ is a closed subgroup of $G$ with Lie algebra $\fg_0$.  Associated to the Schur system is the more restrictive \emph{Schubert system} $\cB_w$ defined by the $G_0$--orbit $B_w \subset \tGr(|w|,T_oX)$ of the $\fg_0$--highest weight line $\fn_w \in R_w$.  In particular, $B_w \subset R_w$.  The corresponding homogeneous bundle $\cB_w \subset \tGr(|w|,TX)$ is precisely the set of $|w|$--planes tangent to (a smooth point of) a Schubert variety of type $w$.  The notions of integrals, rigidity and flexibility for the Schubert system are analogous to those for the Schur system (Section \ref{S:DS}).  Theorem \ref{T:Schur} and $\cB_w \subset \cR_w$ yield

\begin{corollary} \label{C:Schubert}
The homology class $[Y] \subset H_{2|w|}(X,\bZ)$ represented by an integral variety $Y \subset X$ of the Schubert system $\cB_w$ is an integer multiple of the Schubert class $\xi_w$.
\end{corollary}  

For more on the Schubert system see \cite[Section 4]{MR2960030} and the references therein.  The following is due to Bryant (in the case that $X$ is a Grassmannian) and Hong (in general).  

\begin{proposition}\label{P:RvB}
The Schur system $\cR_w$ is rigid if and only if $B_w = R_w$ and the Schubert system $\cB_w$ is rigid.  
\end{proposition}

\begin{remark} \label{R:int}
\begin{a_list}
\item  \emph{A posteriori} the condition that $R_w = B_w$ may be dropped;  it is a consequence of Theorem \ref{T:SchubertFlex} and Corollary \ref{C:SchurFlex} that $\cB_w$ is rigid if and only if $\cR_w$ is rigid.
\item Note that $R_w$ is the intersection of $\tGr(|w|,T_oX)$ with the (projective) linear span of $B_w$ in $\bP(\tw^{|w|}T_oX)$.  In general the containment $B_w \subset R_w$ is strict.  
\end{a_list} 
\end{remark}

The Schubert system lifts to a linear Pfaffian exterior differential system defined on a frame bundle over $X$ \cite[Section 4.4]{MR2960030}.  The lifted system has the advantage that obstructions to rigidity may be identified using Lie algebra cohomology, cf. \cite[Sections 4 and 5]{MR2960030}.  A complete list of the Schubert systems $\cB_w$ for which there exist no cohomological obstructions to rigidity is given by \cite[Theorem 6.1]{MR2960030}: these classes are Schubert rigid.  It is then shown in \cite[Theorem 8.1]{MR2960030} that $R_w = B_w$ for each of these rigid systems.  In particular, the corresponding Schur system $\cR_w$ is rigid.  

The goal of this paper is to prove that these cohomological obstructions are genuine obstructions to the Schubert rigidity of $\xi_w$: that is, given cohomological obstructions there exist nontrivial, algebraic integrals $Y$ of the Schubert system (Theorem \ref{T:SchubertFlex}).  As an integral of $\cB_w$, the variety $Y$ is necessarily an integral of the Schur system $\cR_w$.  In particular, $[Y] = r \xi_w$ for some integer $r > 0$.  

\subsection{The characterization of Schubert varieties} \label{S:aJ}

This section is a brief review of the characterization of Schubert classes $\xi_w$ by an integer $\sfa(w)\ge0$ and a marking $\ttJ(w)$ of the Dynkin diagram.  (The marking is equivalent to a choice of simple roots.)  For more detail see \cite[Section 3.2]{MR2960030}.  Recall \eqref{E:nw}, and let 
$$
  N_w \ = \ \texp(\fn_w) \,.
$$
Then
\begin{equation}\label{E:Xw}
  X_w \ := \ \overline{N_w \cdot o} \ = \ w Y_w
\end{equation}
is a Schubert variety of type $w$ .  (See \cite[Sections 2.3 and 2.4]{MR2960030} for more detail.)  

Let $1 \in W^\fp$ be the identity, and let $w_0 \in W^\fp$ be the longest element.  The $\fg_0$--module $\bI_w$ is trivial if and only if $w \in \{ 1 , w_0\}$.  (The associated Schubert varieties are $X_1 = o$ and $X_{w_0} = X$.)  Assume $w \in W^\fp \backslash \{1,w_0\}$.  Let $\fq_w \subset \fg_0$ be the stabilizer of the highest weight line $\fn_w \in \bP(\tw^{|w|}T_oX)$.  Then there is a subset $\ttJ(w) \subset \{ 1 , \ldots , r \} \backslash\{\tti\}$ with the property that the Lie algebra $\fq_w$ is given by $\fq_w = \fg_{0,\ge0}$, where 
\begin{equation} \label{E:gkl}
  \fg_{k,\ell} \ := \ \{ A \in \fg_k \ | \ [Z_w \,,\, A] = \ell A \} 
  \quad \hbox{ and } \quad Z_w = \sum_{\ttj\in \ttJ(w)} Z_\ttj \, .
\end{equation}
We call $\fg = \oplus \fg_{k,\ell}$ the \emph{$(Z_\tti,Z_w)$--bigraded decomposition} of $\fg$; it is the decomposition of $\fg$ into $(Z_\tti,Z_w)$--eigenspaces.  The following is \cite[Proposition 3.9]{MR2960030}.

\begin{proposition}[{\cite{MR2960030}}]\label{P:aJ}
Let $w \in W^\fp\backslash\{1,w_0\}$.  There exists an integer $\sfa = \sfa(w) \ge 0$ such that $\Delta(w) = \{ \a \in \Delta(\fg_1) \ | \ \a(Z_w) \le \sfa \}$.  Equivalently,
\begin{equation} \label{E:nw_aJ}
  \fn_w \ = \ \fg_{-1,0} \ \op \ \cdots \ \op \ \fg_{-1,-\sfa} \, .
\end{equation}
\end{proposition}

\begin{remark}  
\begin{a_list}
\item 
Since $\xi_w = [X_w]$, and $X_w$ is determined by $\Delta(w)$, the pair $\sfa(w), \ttJ(w)$ characterizes $\xi_w$, where $w \in W^\fp\backslash\{1,w_0\}$.
\item
By \cite[Proposition 3.19]{MR2960030}, the Schubert variety $X_w$ is smooth if and only if $\sfa(w)=0$.  For more on the relationship between the integer $\sfa(w)$ and $\tSing(X_w)$, see \cite{sing}.  
\item
A tableau-esque analog of Proposition \ref{P:aJ} is given by H. Thomas and A. Yong in \cite[Proposition 2.1]{MR2538022}.
\end{a_list}
\end{remark}

A complete list of the $\sfa(w)$, $\ttJ(w)$ that occur is given by \cite[Corollary 3.17]{MR2960030}.   For each of the classical, cominuscule varieties, the relationship between $\sfa(w), \ttJ(w)$ and the familiar geometric descriptions of $Y_w$ is reviewed in the appendix.

\subsection{Divisors} 

In this section we establish a lemma that will be used in the proof of Theorem \ref{T:SchubertFlex}.  Recall the $(Z_\tti,Z_w)$--bigraded decomposition \eqref{E:gkl} of $\fg$, and note that $\fg_{0,0}$ is a reductive subalgebra of $\fg$.

\begin{definition} 
Given a representation $U$ of $\fg_{0,0}$, let $\Pi(U)$ denote the set of highest weights.
\end{definition}

Given $\a \in \Delta^+$, let $r_\a \in W$ denote the corresponding reflection.

\begin{lemma} \label{L:divisor}
Let $\a \in \Delta(w)$.  Then $\Delta(w) = \Delta(w') \sqcup \{\a\}$ for some $w' \in W^\fp$ if and only if $\a \in \Pi(\fg_{1,\sfa})$.  In this case $w' = r_\a w$. 
\end{lemma}

\begin{proof}
By Remark \ref{R:Dw},  there exists $w' \in W^\fp$ such that $\Delta(w)\backslash\{\a\} = \Delta(w')$ if and only if $\Phi = \{\a\} \,\cup\, ( \Delta^+ \backslash \Delta(w))$ is closed.  So, it suffices to show that $\a \in \Pi(\fg_{1,\sfa})$ if and only if $\Phi$ is closed.

Assume $\a \in \Pi(\fg_{1,\sfa})$.  Let $\mu,\nu \in \Phi$, and suppose that $\mu+\nu \in \Delta$.  To see that $\Phi$ is closed, we must show that $\mu+\nu \in \Phi$.  First, note that either $\mu\not=\a$ or $\nu\not=\a$, as $2\a \not\in\Delta$.  Second, if both $\m,\n \in \Delta^+\backslash\Delta(w)$, then $\mu+\nu \in \Delta^+ \backslash\Delta(w) \subset \Phi$, by Remark \ref{R:Dw}.  It remains to consider the case that $\m = \a$ and $\nu \in \Delta^+\backslash\Delta(w)$.  Remark \ref{R:abelian}, and the assumption that $\mu+\nu$ is a root, force $\n \in \Delta^+(\fg_0)$.  Note that $\n(Z_w) \ge 0$.  If $\n(Z_w) = 0$, then $\n$ is a positive root of $\fg_{0,0}$, and the hypothesis that $\a$ is a highest $\fg_{0,0}$--weight implies $\m+\n \not\in\Delta$.  If $\n(Z_w) >0$, then $(\mu+\nu)(Z_w) > \sfa$.  By Proposition \ref{P:aJ}, we have $\Delta^+\backslash\Delta(w) = \Delta^+(\fg_0) \,\cup\, \Delta(\fg_{1,>\sfa})$; in particular, $\mu+\n \in \Delta^+\backslash\Delta(w) \subset \Phi$.

Conversely, suppose that $\Phi$ is closed.  To see that $\a$ is a highest $\fg_{0,0}$--weight, it suffices to show that $\a+\b \not\in\Delta$ for all $\b \in \Sigma(\fg_{0,0})$.  Note that $(\a+\b)(Z_w) = \sfa+0$.  So, if $\a+\b \in\Delta$, then $\a+\b \in \Delta(\fg_{1,\sfa})$.  By Proposition \ref{P:aJ}, $\Delta(\fg_{1,\sfa}) \subset \Delta(w)$.  Therefore, $\a+\b\in\Delta^+\backslash\Phi$; since $\a,\b \in \Phi$, this contradicts the assumption that $\Phi$ is closed.

It remains to show that $w' = r_\a w$.  This is \cite[Proposition 3.2.15(3)]{MR2532439}.
\end{proof}

\section{Flexibility} \label{S:flex}

In \cite[Section 4]{MR2960030} we saw that a Lie algebra cohomology $H^1(\fn_w,\fg_w^\perp)$ group is naturally associated with the Schubert system $\cB_w$.  The cohomology group contains two distinguished subspaces $H^1_{1,\sfa-1}$ and $H^1_{2,2\sfa-1}$.  When the subspaces are trivial, the Schubert system is necessarily rigid, \cite[Theorem 5.38]{MR2960030}.  The two subspaces are trivial precisely when two representation theoretic conditions, $\sfH_1 = \sfH_1(w)$ and $\sfH_2 = \sfH_2(w)$, are satisfied (Definition \ref{D:H1H2}).  When both conditions hold, we say that $\sfH_+$ is satisfied.  In summary: define
$$
  \cO_w \ = \ H^1_{1,\sfa-1}\,\op\,H^1_{2,2\sfa-1} \,;
$$
then 
\begin{center}
  $\cO_w = 0$ if and only if $\sfH_+$ holds. 
\end{center}

\begin{theorem} \label{T:SchubertFlex}
Let $X = G/P$ be a cominuscule variety, and fix $w\in W^\fp\backslash\{1,w_0\}$.  The Schubert system $\cB_w$ is rigid if and only if $\cO_w=0$.  Moreover, when $\cO_w\not=0$, there exist nontrivial algebraic integrals; in particular, the Schubert class $\xi_w$ is Schubert rigid if and only if $\cO_w=0$.
\end{theorem}

\begin{remark}[Geometric interpretation]
Recall (Section \ref{S:schubert}) that an irreducible variety $Y \subset X$ is an integral of the Schubert system $\cB_w$ if and only if at each smooth point $y \in Y$ there exists a Schubert variety $Y' = Y'(y)$ of type $w$, such that $y$ is also a smooth point of $Y'$, and $T_yY = T_{y} Y'$.  Equivalently, if $\cY$ denotes the canonical lift of the smooth locus $Y_0 \subset Y$ to $\tGr(|w|,TX)$, then $\cY$ and $\cY'$ intersect at $\tilde y = T_y Y \in \tGr(|w|,TX)$.  If the condition $\sfH_1$ holds, then $\cY$ and $\cY'$ are necessarily tangent at $\tilde y$; and when both conditions $\sfH_1$ and $\sfH_2$ hold, $\cY$ and $\cY'$ agree to second-order at $\tilde y$.
\end{remark}

\begin{corollary} \label{C:SchurFlex}
Let $w \in W^\fp\backslash\{1,w_0\}$.  The Schur $\cR_w$ system is rigid if and only if $\cO_w=0$.  Moreover, when $\cO_w\not=0$, there exist nontrivial algebraic integrals; in particular, the Schubert class $\xi_w$ is Schur rigid if and only if $\cO_w=0$.
\end{corollary}

\begin{proof}
By \cite[Theorem 8.1]{MR2960030} if $\cO_w=0$, then $B_w = R_w$.  The corollary then follows from  Proposition \ref{P:RvB} and Theorem \ref{T:SchubertFlex}.
\end{proof}

\begin{remark}
A complete list of the Schubert varieties (in an irreducible cominuscule $X$) satisfying $\sfH_+$ is given by \cite[Theorem 6.1]{MR2960030}.
\end{remark}

\begin{corollary}[Poincar\'e duality]\label{C:PD}
A Schubert class $\xi_w$ is Schur rigid if and only if its Poincar\'e dual $\xi_{w^*}$ is Schur rigid.
\end{corollary}

\begin{proof}
Given $w \in W^\fp$, let $w^* \in W^\fp$ be the element associated to the Poincar\'e dual of $\xi_w$.  By \cite[Corollary 6.2]{MR2960030}, the condition $\sfH_+(w)$ holds if and only if $\sfH_+(w^*)$ holds.  It follows from Theorem \ref{T:SchubertFlex} (resp. Corollary \ref{C:SchurFlex}) that the Schubert system $\cB_w$ (resp. the Schur system $\cR_w$) is rigid if and only if the system $\cB_{w^*}$ (resp. $\cR_{w^*}$) is rigid. 
\end{proof}

As given by \cite[Definitions 5.27 and 5.36]{MR2960030}, we have 

\begin{definition} \label{D:H1H2}
The condition \emph{$\sfH_1$ fails} when there exist $\fg_{0,0}$--highest weights $-\b \in \Delta(\fg_{0,-1})$ and $\c \in \Delta(\fg_{1,\sfa})$ such that 
\begin{equation} \label{E:H1}
  \left[ \fg_{-\b} \,,\, \fg_\c \right] \ \stackrel{\mathrm{(i)}}{=} \ \{0\} 
  \ \not= \ 
  \left[ \fg_{\b} \,,\, \fg_\c \right] \ \stackrel{\mathrm{(ii)}}{=} \
  \left[ \fg_{\b} \,,\, \fg_{1,\sfa} \right] \, .
\end{equation}
The condition \emph{$\sfH_2$ fails} when there exist $\fg_{0,0}$--highest weights $\e \in \Delta(\fg_{1,\sfa-1})$ and $\c \in \Delta(\fg_{1,\sfa})$ such that 
\begin{equation} \label{E:H2}
  \{0\} \ \stackrel{\mathrm{(i)}}{\not=} \ \left[ \fg_{\e} \,,\, \fg_{-\c} \right] 
  \ \stackrel{\mathrm{(ii)}}{=} \ \left[ \fg_{\e} \,,\, \fg_{-1,-\sfa} \right] \, .
\end{equation}
When both $\sfH_1$ and $\sfH_2$ hold, we say that \emph{the condition $\sfH_+$ is satisfied}.
\end{definition} 

\begin{remark}
Note that $-\b$ is a $\fg_{0,0}$--highest weight if and only if $\b$ is a $\fg_{0,0}$--lowest weight.  That is, $\b\in\Delta(\fg_{0,1})$ is a simple root.
\end{remark}

\begin{remark} \label{R:a=0}
Note that $\sfH_2$ is trivially satisfied when $\sfa=0$, as $\fg_{1,\sfa-1} = \fg_{1,-1} = 0$.
\end{remark}

\begin{example} 
In Example \ref{eg:Gr0}, the $\fg_{0,0}$--lowest weights $\b\in\Delta(\fg_{0,1})$ are $\{ \a_2 \,,\, \a_3 \,,\, \a_7 \,,\, \a_9\,,\, \a_{12}\}$, and the $\fg_{0,0}$--highest weights are
\begin{eqnarray*}
  \Pi(\fg_{1,\sfa}) & = &
  \{ \a_1+\cdots+\a_6 \,,\ \a_3+\cdots+\a_8 \,,\ \a_4+\cdots+\a_{11} \} \,,\\
  \Pi(\fg_{1,\sfa-1}) & = & 
  \{ \a_3+\cdots+\a_6 \,,\ \a_4+\cdots+\a_8 \}\,.
\end{eqnarray*}
The condition $\sfH_1$ fails for $\b = \a_9$ and $\c = \a_3 + \cdots + \a_8$.  The condition $\sfH_2$ is satisfied.
\end{example}

\begin{example} 
In Example \ref{eg:LG0}, the $\fg_{0,0}$--lowest weights $\b\in\Delta(\fg_{0,1})$ are $\{ \a_1 \,,\, \a_2 \,,\, \a_4 \}$, and the $\fg_{0,0}$--highest weights are
\begin{eqnarray*}
  \Pi(\fg_{1,\sfa}) & = &
  \{ \a_1+\cdots+\a_5 \,,\ \a_2+2\a_3+2\a_4+\a_5 \} \,,\\
  \Pi(\fg_{1,\sfa-1}) & = & 
  \{ \a_2+\cdots+\a_5 \,,\ 2\a_3+2\a_4+\a_5 \}\,.
\end{eqnarray*}
The condition $\sfH_1$ fails for $\b = \a_4$ and $\c = \a_1 + \cdots + \a_5$.  The condition $\sfH_2$ fails for $\e=2\a_3+2\a_4+\a_5$ and $\c=\e+\a_2$.
\end{example}

\subsection{Outline of the proof of Theorem \ref{T:SchubertFlex}} \label{S:outline}
By  \cite[Theorem 5.38]{MR2960030}, the Schubert system $\cB_w$ is rigid when $\sfH_+$ is satisfied.  So to prove Theorem \ref{T:SchubertFlex} if suffices to show that there exist nontrivial integrals $Y$ of the Schubert system when $\sfH_+$ fails.  This is done in Sections \ref{S:prfH1} and \ref{S:prfH2}.  We will see that there exists a Schubert divisor $X_{w'} \subset X_w$ and a 1-parameter subgroup $A \subset G$ with the property that $Y = \overline{A\cdot X_{w'}}$ is a nontrivial integral variety of the Schubert system $\cB_w$.

To see that the varieties $Y$ constructed in Sections \ref{S:prfH1} and \ref{S:prfH2} are nontrivial integrals of the Schubert system we must review some of the results of \cite{MR2960030}.  Define
$$
  \fn_w^\perp \ = \ \bigoplus_{\a\in\Delta(\fg_1)\backslash\Delta(w)} \fg_{-\a}
  \ = \ \bigoplus_{b>\sfa} \fg_{-1,-b} \,,
$$
so that $\fg_{-1} = \fn_w \op \fn_w^\perp$, and let
$$
  \fg_w^\perp \ = \ \fn_w^\perp \,\op\, \fg_{0,-} \,\op\, \fg_{1,<\sfa} \, .
$$
Each cohomology class $[\n] \in H^1(\fn_w,\fg_w^\perp)$ admits a unique harmonic representative $\n \in \cH^1\subset \fg_w^\perp \ot \fn_w^*$ by \cite[Proposition 5.10]{MR2960030}.  The grading element $Z_\tti$ induces a graded decomposition $\cH^1 = \cH^1_0 \op \cH^1_1 \op \cH^1_2$ with $\cH_0 \subset \fn_w^\perp \ot \fn_w^*$, $\cH^1_1 \subset \fg_{0,-} \ot \fn_w^*$ and $\cH^1_2 \subset \fg_{1,<\sfa}\ot\fn_w^*$.

Let $\vartheta$ denote the $\fg$--valued, left-invariant Maurer-Cartan form on $G$.  Set $\fg_w = \fn_w \op \fg_{0,\ge0} \op \fg_{1,\ge\sfa}$, so that $\fg = \fg_w \op \fg_w^\perp$.  Let $\vartheta_{k,\ell}$ denote the $\fg_{k,\ell}$--valued component of the Maurer-Cartan form, with respect to the direct sum decomposition $\fg = \op \fg_{k,\ell}$ of \eqref{E:gkl}.  Similarly, let $\vartheta_{\fg_w}$ and $\vartheta_{\fn_w^\perp}$ respectively denote the $\fg_w$-- and $\fn_w^\perp$--valued components of $\vartheta$, with respect to the direct sum decompositions $\fg = (\fn_w \op \fn_w^\perp) \op \fg_{\ge0} = \fg_w \op \fg_w^\perp$.  The following is an amalgam of \cite[Lemma 4.9 and Corollary 5.12]{MR2960030}. 

\begin{lemma}\label{L:hrm}
There is a bijective correspondence between submanifolds $U \subset G$ such that 
\begin{i_list}
  \item $\vartheta_{\fg_w} : T_gU \to \fg_w$ is a linear isomorphism 
        for all $g \in U$,
  \item $\vartheta_{\fn_w^\perp} = 0$ vanishes on $U$,
  \item there exists a smooth map $\lambda : U \to \cH^1_1$ such that 
        $\vartheta_{0,-} = \lambda(\vartheta_{\fn_w})$ on $U$,
\end{i_list} 
and integral manifolds $U\cdot o \subset X$ of the Schubert system.  Additionally, there exists a smooth function $\mu : U \to \fg_{1,<\sfa} \ot \fn_w^*$ such that $\vartheta_{1,<\sfa} = \mu(\vartheta_{\fn_w})$ on $U$.  If $\lambda$ is identically zero, then $\mu$ takes values in $\cH^1_2$. 
\end{lemma}

\noindent The submanifold $U$ of Lemma \ref{L:hrm} is the sub-bundle $\cF^0$ of \cite[Corollary 5.12]{MR2960030}.  It is a consequence of the arguments of \cite[Section 4.7]{MR2960030} that the sub-bundle $\cF^0$ is unique.  Then  \cite[Proposition 4.13]{MR2960030} yields

\begin{lemma} \label{L:hrm0}
The integral manifold $U\cdot o \subset X$ is a Schubert variety of type $w$ if and only if both $\lambda$ and $\mu$ vanish identically on $U$.  
\end{lemma}

The integrals $Y = \overline{A \cdot X_{w'}}$ constructed in Sections \ref{S:prfH1} and \ref{S:prfH2} will be of the form $Y = \overline{U\cdot o}$, where $U$ satisfies the hypothesis of Lemma \ref{L:flex}.

\begin{lemma} \label{L:flex}
Let $U \subset G$ be a submanifold such that 
\begin{i_list}
  \item $\vartheta_{\fg_w} : T_gU \to \fn_w$ is a linear isomorphism 
        for all $g \in U$,
  \item $\vartheta_{\fn_w^\perp}$ vanishes on $U$.
\end{i_list}
Then $U \cdot o \subset X$ is a \emph{nontrivial} integral manifold of the Schubert system if either
\begin{a_list}
\item there exists a smooth map $\lambda : U \to \cH^1_1$, that is not identically zero, such that $\vartheta_{0,-} = \lambda(\vartheta_{\fn_w})$ on $U$; or
\item $\vartheta_{0,-}$ vanishes on $U$, and there exists a smooth map $\mu : U \to \cH^1_2$, that is not identically zero, such that $\vartheta_{1,<\sfa} = \mu(\vartheta_{\fn_w})$.
\end{a_list}
\end{lemma}

\begin{proof}
Let $G' \subset G$ be the connected Lie group with Lie algebra $\fg_{0,\ge0} \op \fg_{1,\ge\sfa}$.  Let $U \subset G$ be a submanifold satisfying (i) and (ii) of Lemma \ref{L:flex}.  Then $U' = U G' \subset G$ is a submanifold satisfying (i) and (ii) of Lemma \ref{L:hrm}.  

Suppose $U$ satisfies Lemma \ref{L:flex}(a).  Then $U'$ satisfies Lemma \ref{L:hrm}(iii).  Lemma \ref{L:hrm0} implies that $U\cdot o = U'\cdot o \subset X$ is a nontrivial integral of the Schubert system.

Suppose $U$ satisfies Lemma \ref{L:flex}(b).  The $\vartheta_{0,-}$ also vanishes on $U'$.  In particular, $U'$ satisfies Lemma \ref{L:hrm}(iii) trivially.  Moreover, since $\lambda$ is identically zero on $U'$, the function $\mu$ of Lemma \ref{L:hrm} necessarily takes values in $\cH^1_2$.  Since $\mu$ is nonzero on $U \subset U'$, Lemma \ref{L:hrm0} implies that $U\cdot o = U'\cdot o \subset X$ is a nontrivial integral of the Schubert system.
\end{proof}

\subsection{Proof of Theorem \ref{T:SchubertFlex} in the case that $\sfH_1$ fails} \label{S:prfH1}

We will follow the strategy outlined in Section \ref{S:outline}.  Suppose that $\sfH_1$ fails for a pair $\c$, $-\b$; see Definition \ref{D:H1H2}.  Then \cite[Lemma 5.28]{MR2960030} asserts that
\begin{equation} \label{E:cH1}
  \fg_{-\b}\ot \fg_\c \subset \cH^1_1 \,.
\end{equation}

Lemma \ref{L:divisor} and the fact that $\c \in \Delta(\fg_{1,\sfa})$ is a $\fg_{0,0}$--highest weight yield $w' \in W^\fp$ such that $\Delta(w') \sqcup \{\c\} = \Delta(w)$.  Fix $0\not=C \in \fg_{-\c}$ and $0 \not= B \in \fg_{-\b}$.  Let $\fa = \langle B + C \rangle$.

\begin{claim*} 
The direct sum $\fn' := \fa \op \fn_{w'}$ is a subalgebra of $\fg$.
\end{claim*}

The claim is proved at the end of this section.  Assuming the claim holds, let $N' = \texp(\fn') \subset G$.  When restricted to $N'$ the Maurer-Cartan form $\vartheta$ takes values in $\fn'$.  In particular, $U = N'$ satisfies hypotheses (i) and (ii) of Lemma \ref{L:flex}.  Moreover, the map $\lambda$ of Lemma \ref{L:flex} is the nonzero, constant map given by $\lambda_{|\fn_{w'}} =0$ and $\lambda(C) = B$.  By \eqref{E:cH1}, $\lambda$ lies in $\cH^1_1$.  It follows from Lemma \ref{L:flex}.a that $N'\cdot o \subset X$ is a nontrivial integral manifold of the Schubert system.

Finally, note that $\fn' \subset \fg_{-1} \op \fg_{0,-1}$ is nilpotent.  It follows that $N' = \texp(\fn')$ is a linear algebraic subgroup of $G$.  Therefore, the Zariski closure $Y_1 = \overline{N' \cdot o}$ is an irreducible algebraic subvariety of $X$ with $\tdim\,Y_1 = \tdim\,N' = \tdim\,X_w$.  Therefore, $Y_1$ is a nontrivial, integral variety of the Schubert system.  Modulo the claim, this completes the proof of Theorem \ref{T:SchubertFlex} in the case that $\sfH_1$ fails.

\begin{remark}
Let $A = \texp(\fa)$.  Then $\overline{N'} = \overline{A\,N_{w'}}$ and $Y_1 = \overline{A \cdot X_{w'}}$.
\end{remark}

\noindent The proof the the claim makes use of the following

\begin{remark} \label{R:roots}
Given two roots $\m,\n\in\Delta$, we have $[\fg_\m , \fg_\n] = \fg_{\m+\n}$.  In particular, $\m+\n\in\Delta$ if and only if $[\z,\xi] \not=0$ for every $0\not=\z \in \fg_\m$ and $0\not=\xi\in\fg_\n$.  See, for example, \cite[Corollary 2.35]{MR1920389}.
\end{remark}

\begin{proof}[Proof of claim]
Since both $\fa$ and $\fn_{w'}$ are subalgebras of $\fg$, it suffices to show that $[\fa \,,\, \fn_{w'} ] \subset \fn'$.  To that end, let $\xi \in \fn_{w'}$.  Then $C \in \fg_{-\c} \subset \fg_{-1}$ and Remark \ref{R:abelian} yield
$$
  [ B+C \,,\, \xi ] \ = \ [ B \,, \, \xi ] \, .
$$
Without loss of generality we may assume that $\xi \in \fg_{-\m}$ is a nonzero root vector, $\m \in \Delta(w')$.  By Remark \ref{R:roots}, $[B\,,\, \x] \in \fg_{-\b-\m}$ is nonzero if and only if $\b+\m$ is a root; assume this is the case.  We wish to show that $\fg_{-\b-\m} \subset \fn_{w'}$; equivalently, $\b+\m \in \Delta(w')$.  

By Proposition \ref{P:aJ}, $\Delta(w) = \Delta(\fg_{1,\le\sfa})$.  Since $\Delta(w') \sqcup\{\c\} = \Delta(w)$ and $\c\in\Delta(\fg_{1,\sfa})$, we have 
$$
  0 \le \m(Z_w) \le \sfa \quad \hbox{ and } \quad
  \Delta(\fg_{1,<\sfa}) \ \subset \ \Delta(w') 
  \ \subset \ \Delta(\fg_{1,\le \sfa}) \, .
$$ 
We consider three cases:  First, if $\m(Z_w) < \sfa-1$, then $(\m+\b)(Z_w) < \sfa$ and $\b+\m\in\Delta(w')$.  Second, if $\m(Z_w) = \sfa-1$, then $(\m+\b)(Z_w) = \sfa$.  By (\ref{E:H1}.i), $\m+\b \not=\c$, so $\m+\b \in \Delta(w')$.  Third, if $\m(Z_w) = \sfa$, then (\ref{E:H1}.ii) implies $\b+\m$ is not a root.  
\end{proof}

\subsection{Proof of Theorem \ref{T:SchubertFlex} in the case that $\sfH_2$ fails} \label{S:prfH2}

We will follow the strategy outlined in Section \ref{S:outline}.  Suppose that $\sfH_2$ fails for a pair $\c \in \Delta(\fg_{1,\sfa})$ and $\e \in\Delta(\fg_{1,\sfa-1})$, see Definition \ref{D:H1H2}. Then \cite[Lemma 5.39]{MR2960030} asserts that
\begin{equation} \label{E:cH2}
  \fg_{\e} \ot \fg_\c \ \subset \ \cH^1_2 \, .
\end{equation}

By Lemma \ref{L:divisor} there exists $w' \in W^\fp$ such that $\Delta(w') \sqcup \{\c\} = \Delta(w)$.  Fix $0\not=C \in \fg_{-\c}$ and $0 \not= E \in \fg_\e$.  Set 
\begin{equation} \label{E:2A}
  \fa \ := \ \langle C+E \rangle \ \subset \ \fg
  \qquad\hbox{and}\qquad
  A \ := \ \texp(\fa) \ \subset \ G \, .
\end{equation}
We will show that $Y_2 = \overline{ A N_{w'} \cdot o }$ is a nontrivial integral variety of the Schubert system.

Because $\texp : \fg \to G$ is a diffeomorphism from a neighborhood of $0 \in \fg$ to a neighborhood of $\tId \in G$, it follows that there exist connected neighborhoods $A_0 \subset A$ and $N_0 \subset N_{w'}$ of $\tId$ that are embedded submanifolds in $G$.  Let $\mathbf{m} : G \times G \to G$ denote the multiplication map.  Then $\mathbf{m}_* : \fg\times\fg \to \fg$ is given by $\mathbf{m}_*(u,v) = u+v$.  In particular, the restriction $\mathbf{m}_*: \fa \times \fn_{w'} \to \fa \op \fn_{w'}$ is a bijection.  It follows, shrinking $A_0$ and $N_0$ if necessary, that there is a neighborhood $U \subset G$ of $\tId$ such that $\mathbf{m} : A_0 \times N_0 \to U$ is an isomorphism.

Compute the Maurer-Cartan form $\vartheta$ on $U \subset G$ as follows.  Let $g_0 g_1 \in U$ with $g_0 \in A_0$ and $g_1 \in N_0$.  The isomorphism $\phi : G \to G$ defined by $\phi(g) = g_0 g g_1$ preserves $A \,N_{w'}$ and so induces an isomorphism $\phi_* : \fa\op\fn_{w'} \to T_{g_0g_1} U$.  So any element of $T_{g_0g_1} U$ may be expressed as $\phi_* \xi$ with $\xi \in \fa\op\fn_{w'}$.    Given $g \in G$, let $L_g , R_g : G \to G$ denote the left- and right-multiplication maps.  Then $\phi = L_{g_0}\,R_{g_1}$.  By definition 
\begin{subequations}\label{SE:4}
\begin{equation} \label{E:4a}
  \vartheta (\phi_*\xi) \ = \ (L_{g_0g_1}^{-1})_*( \phi_* \xi)
  \ = \ \tAd_{g_1^{-1}} \xi \, .
\end{equation}
Therefore, $\vartheta_{|U}$ takes values in $\tAd_{N_{w'}}(\fa \op \fn_{w'}) = \tAd_{N_{w'}}(\fa) \op \tAd_{N_{w'}}(\fn_{w'})$.  Since $\fn_{w'}$ is an algebra, we have 
\begin{equation} \label{E:4b}
  \tAd_{N_{w'}}(\fn_{w'}) \ = \ \fn_{w'} \,.
\end{equation}  
It remains to consider
\begin{equation} \label{E:4c}
  \tAd_{N_{w'}} (\fa) \ = \ \tAd_{\texp(\fn_{w'})} (\fa)
   \ = \ (\texp\circ\tad_{\fn_{w'}}) (\fa)  \, .
\end{equation}
The algebra $\fa$ is spanned by $C+E$.  By Remark \ref{R:abelian}, $[\fn_{w'} , C+E] = [ \fn_{w'} , E]$.  Equations (\ref{E:nw_aJ}, \ref{E:H2}), and $\fn_{w'} \op \fg_{-\c} = \fn_w$ imply that $[\fn_{w'} \,,\,E] \ \subset \ \fg_0$.  Moreover, the fact that $\e \in \Delta(\fg_{1,\sfa-1})$ is a highest $\fg_{0,0}$--weight yields
$$
  [\fn_{w'} \,, \, E] \ \subset \ \fh_\e \op \fg_0^+ \,,
  \quad\hbox{where } \ \fh_\e := [\fg_\e\,,\,\fg_{-\e}] \ \hbox{ and } \ 
  \fg_0^+ := \op_{\a\in\Delta^+(\fg_0)} \fg_\a \, .
$$
By Proposition \ref{P:aJ},  $\Delta(w')\sqcup\{\c\} = \Delta(w) = \Delta(\fg_{1,\le\sfa})$.  So, since $\c\in\Delta(\fg_{1,\sfa})$ is a highest $\fg_{0,0}$--weight, we have $[ \fn_{w'} , \fh_\e\op\fg_0^+ ] \subset \fn_{w'}$.  Therefore, 
\begin{equation} \label{E:4d}
  (\texp\circ\tad_{\fn_{w'}})(\fa) \ \subset \ \fa \ \op \ 
  (\fh_\e \op \fg_0^+) \ \op \ \fn_{w'} \, .
\end{equation}
\end{subequations}
Now \eqref{SE:4} yields $\vartheta_{\fn_w^\perp} = 0$ on $U$, and $\vartheta_{\fn_w} : T_zU \to \fn_w$ is a linear isomorphism for all $z \in U$.  Lemma \ref{L:flex} implies $U\cdot o \subset X$ is an integral manifold of the Schubert system.  Equations \eqref{SE:4} also imply that $\vartheta_{0,-}$ vanishes on $U$.  Similarly, \eqref{SE:4} implies that $\vartheta_{1,<\sfa} = \m(\vartheta_{\fn_w})$, where $\mu: U \to \fg_{1,<\sfa}\ot \fn_w^*$ is the nonzero, constant map defined by $\mu_{|\fn_{w'}} = 0$ and $\mu(C) = E$.  By \eqref{E:cH2}, $\mu$ takes values in $\cH^1_2$.  By Lemma \ref{L:flex}(b), $U \cdot o \subset X$ is a nontrivial integral manifold of the Schubert system.

Finally, to see that the Zariski closure 
\begin{equation} \label{E:Y2}
  Y_2 \ = \ \overline{U \cdot o} \ = \ \overline{A N_{w'} \cdot o}
\end{equation}
is a variety of dimension $|w|$, it suffices to observe that $A$ and $N_{w'}$ are algebraic.  This is a consequence of the fact that both $\fa$ and $\fn_{w'}$ are nilpotent.  The nilpotency of $\fn_{w'}$ is immediate from Remark \ref{R:abelian}.  To see that $\fa$ is nilpotent, note that $[\fa , \fg_{\pm1} ] \subset \fg_0$ and $\tad_{\fa}^2(\fg_{0,\ell}) \subset \fg_{0,\ell-1}$; these relations imply that the adjoint action of $\fa$ on $\fg$ is nilpotent.  This completes the proof of Theorem \ref{T:SchubertFlex}.

\appendix
\section{Geometric versus representation theoretic descriptions} \label{S:appendix}

In this section we provide the `dictionary' between the representation theoretic $(\sfa,\ttJ)$--description of Schubert varieties (Section \ref{S:aJ}) used in the proof of Theorem \ref{T:SchubertFlex} and the more familiar and geometric, partition-based descriptions for the classical $G/P$.  This dictionary, applied to Corollary \ref{C:SchurFlex} and \cite[Theorem 6.1]{MR2960030}, yields the theorems of Section \ref{S:statements}.

\subsection{Notation}
Given a vector space $V \simeq\bC^n$, we fix a basis $\{ e_1 , \ldots , e_n\}$.  Let $\{ e^1 , \ldots , e^n\}$ denote the dual basis of $V^*$.  Set 
$$
  e^k_\ell \ \dfn \  e_\ell \ot e^k \ \in \ \tEnd(V)
  \quad\hbox{for all} \quad \le k,\ell \le n \,.
$$

\begin{remark}\label{R:newaJ}
I follow the notation of \cite{MR2960030}, with the following exception.  In \cite{MR2960030}, we uniformly write $\ttJ = \{\ttj_1<\cdots<\ttj_\sfp\}$.  In this paper, it is convenient to reorder the $\ttj_\ell$ in some cases.
\end{remark}

\subsection{Odd dimensional quadrics \boldmath $Q^{2n-1} = B_n/P_1$ \unboldmath } \label{S:aJodd}
Set $m = 2n-1$.  There is a bijection between $W^\fp\backslash\{1,w_0\}$ and pairs $\sfa,\ttJ$ such that $\ttJ = \{ \ttj \} \subset \{2,\ldots,n\}$ and $\sfa \in \{0,1\}$; see \cite[Corollary 3.17]{MR2960030}.  The Schubert variety $X_w$ associated with a fixed $\sfa , \ttJ$ may be described as follows.  

Given a nondegenerate symmetric quadric form $(\cdot,\cdot)$ on $\bC^{2n+1}$, fix a basis $\{ e_1 , \ldots , e_{2n+1}\}$ of $\bC^{2n+1}$ so that $(e_k , e_\ell) = (e_{n+k} , e_{n+\ell}) = (e_k , e_{2n+1} ) = (e_{n+k} , e_{2n+1}) = 0$, $(e_k , e_{n+\ell} ) = \d_{k\ell}$, for all $1 \le k,\ell\le n$, and $(e_{2n+1} , e_{2n+1})=1$.  The abelian subalgebra $\fg_{-1}$ is spanned by the root vectors
\begin{eqnarray*}
  e^1_k - e^{n+k}_{n+1}\,, & \hbox{with root} & -(\a_1 + \cdots+\a_{k-1}),\\
  e^1_{n+k} - e^{k}_{n+1}\,, & \hbox{with root} & 
  -\left( \a_1 + \cdots + \a_{k-1} + 2(\a_k + \cdots + \a_n) \right),\\
  e^1_{2n+1} - e_{n+1}^{2n+1}\,, & \hbox{with root} 
  & -(\a_1 + \cdots + \a_{n})\,,
\end{eqnarray*}
$2 \le k \le n$.  

Set $o = [e_1] \in \bP^{2n}$.  If $\sfa = 0$, then 
$$
  \fn_w = \langle e^1_k- e^{n+k}_{n+1} \ | \ k \le \ttj \rangle
  \quad\hbox{ and }\quad
  X_w = \overline{\texp(\fn_w) \cdot o }= 
  \bP \langle e_1 , \ldots,e_\ttj \rangle = \bP^{\ttj-1}\,.  
$$
If $\sfa=1$, then 
$$
  \fn_w = \langle e^1_k- e^{n+k}_{n+1} \,,\ e^1_{n+\ell} - e^{\ell}_{n+1} \,,\
   e^1_{2n+1} - e_{n+1}^{2n+1} \ | \ 1\le k\le n \,,\ \ttj < \ell \rangle
$$
and $X_w = \overline{\texp(\fn_w)\cdot o} =  Q^m \cap \bP\langle e_1 , \ldots , e_{n+1} , e_{n+\ttj+1} , \ldots , e_{2n+1} \rangle$.

\subsection{Even dimensional quadrics \boldmath $Q^{2n-2} = D_n/P_1$\unboldmath} \label{S:aJeven}
Set $m = 2n-2$.  There is a bijection between $W^\fp\backslash\{1,w_0\}$ and pairs $\sfa,\ttJ$ such that either
\begin{circlist}
\item  $\sfa=0$ and $\ttJ = \{ \ttj\} \subset \{2,\ldots,n\}$ or $\ttJ = \{n-1,n\}$; or
\item $\sfa=1$ and $\ttJ = \{\ttj\} \subset \{2,\ldots,n-2\}$ or $\ttJ = \{n-1,n\}$.
\end{circlist} 
See \cite[Corollary 3.17]{MR2960030}.  The Schubert variety $X_w$ associated with a fixed $\sfa , \ttJ$ may be described as follows.  

Given a nondegenerate symmetric quadric form $(\cdot,\cdot)$ on $\bC^{2n}$, fix a basis $\{ e_1 , \ldots , e_{2n}\}$ of $\bC^{2n}$ so that $(e_k , e_\ell) = (e_{n+k} , e_{n+\ell}) = 0$ and $(e_k , e_{n+\ell} ) = \d_{k\ell}$, for all $1 \le k,\ell\le n$.  The abelian subalgebra $\fg_{-1}$ is spanned by the root vectors
\begin{eqnarray*}
  e^1_k - e^{n+k}_{n+1}\,, & \hbox{with root} & -(\a_1 + \cdots+\a_{k-1}),\\
  e^1_{n+\ell} - e^{\ell}_{n+1}\,, & \hbox{with root} & 
  -\left( \a_1 + \cdots + \a_{\ell-1} 
  + 2(\a_\ell + \cdots + \a_{n-2}) +\a_{n-1}+\a_n\right),\\
  e^1_{2n} - e^{n}_{n+1}\,, & \hbox{with root} 
  & -(\a_1 + \cdots+\a_{n-2}) - \a_n,
\end{eqnarray*}
with $2 \le k \le n$ and $2\le \ell \le n-1$.  

Set $o = [e_1] \in \bP\bC^{2n}$.  First suppose that $\sfa=0$.  If $\ttJ = \{ \ttj\}$ with $2 \le \ttj \le n-2$, then $X_w = \bP^{\ttj-1}$.  If $\ttJ=\{n-1\}$ or $\ttJ=\{n\}$, then $X_w = \bP^{n-1}$.  If $\ttJ = \{ n-1 , n \}$, then $X_w = \bP^{n-2}$.

Next suppose that $\sfa=1$.  If $\ttJ = \{ \ttj \}$ with $2 \le \ttj \le n-2$, then $X_w = Q^m \cap \bP \langle e_1 , \ldots , e_{n+1} , e_{n+\ttj+1} , \ldots , e_{2n}  \rangle$.  If $\ttJ = \{n-1,n\}$, then $X_w = Q^m \cap \bP\langle e_1 , \ldots , e_{n+1} , e_{2n} \rangle$.

\subsection{Grassmannians $\tGr(\tti,n+1) = A_n/P_\tti$} \label{S:aJA}
There is a bijection between $W^\fp\backslash\{1,w_0\}$ and pairs $\sfa,\ttJ$ such that $\ttJ = \{ \ttj_\sfp \,,\,\ldots \,,\, \ttj_1 \,,\, \ttk_1 \,,\, \ldots \,,\, \ttk_\sfq \} \subset \{1,\ldots,n\} \backslash\{\tti\}$ is ordered so that 
$$
  1 \ \le \ \ttj_\sfp \,<\,\cdots \,<\, \ttj_1 \,< \ \tti \ < 
  \, \ttk_1 \,<\, \cdots \,<\, \ttk_\sfq \ \le \ n \,,
$$
and satisfying $\sfp,\sfq \in \{\sfa, \sfa+1\}$; see \cite[Corollary 3.17]{MR2960030}.  (Beware, these $\sfp,\sfq$ do not agree with those of \cite{MR2960030}, cf. Remark \ref{R:newaJ}.)  For convenience we set
$$
  \ttj_{\sfp+1} \, := \, 0 \, , \quad \ttj_0 \, := \, \tti \, =: \, \ttk_0 \,, 
  \quad \ttk_{\sfq+1} \, := \, n+1 \, .
$$

Given $\sfa,\ttJ$, the corresponding Schubert variety $X_w \subset \tGr(\tti,n+1)$ is described as follows.  Fix a basis $\{e_1,\ldots , e_{n+1} \}$ of $\bC^{n+1}$.  The abelian subalgebra $\fg_{-1}$ is spanned by the root vectors $\{ e^k_\ell \ | \ 1 \le k \le \tti < \ell \le n+1\}$; the corresponding roots are $-(\a_{k}+\cdots+\a_{\ell-1})$.   Define a filtration $\ttF_{\sfa+1} \subset \ttF_\sfa \subset \cdots \subset \ttF_1 \subset \ttF_0$ of $\bC^{n+1}$ by 
\begin{equation} \nonumber
  \ttF_\ell \ = \ \langle e_1 , e_2 , \ldots , e_{\ttj_\ell} \,,\, 
  e_{\tti+1} , e_{\tti+2} , \ldots , e_{\ttk_m} \rangle 
  \quad\hbox{with}\quad \ell+m = \sfa+1 \, .
\end{equation}
Set $o = [e_1 \wedge \cdots \wedge e_\tti ] \in \bP (\tw^\tti \bC^{n+1})$.  Then
\begin{equation}  \label{E:Xw_app}
  X_w \ = \ \{ E \in \tGr(\tti,n+1) \ | \ 
  \tdim( E \cap \ttF_\ell) \ge \ttj_{\ell} \,,
  \ 0 \le \ell \le \sfa+1 \} \,.
\end{equation}

\begin{example} \label{eg:Gr0}
Consider $X = \tGr(5,13) \simeq A_{12}/P_5$.  The marking $\ttJ = \{ 2,3,7,9,12 \}$ and integer $\sfa=2$ define the filtration $\ttF_3 \subset \ttF_2 \subset \ttF_1 \subset \ttF_0$ as
\begin{eqnarray*}
  \ttF_3 & = & \langle 0 \rangle \,, \quad 
  \ttF_2 \ = \ \langle e_1,e_2 \,,\, e_6 , e_7 \rangle \,,\\
  \ttF_1 & = & \langle e_1,e_2,e_3\,,\, e_6,e_7,e_8,e_9 \rangle \,,\quad
  \ttF_0 \ = \ \langle e_1 ,\ldots,e_5 \,,\, e_6 , \ldots e_{12} \rangle \, .
\end{eqnarray*}
The associated Schubert variety is the set of all $E \in \tGr(5,13)$ such that 
$$
  \tdim( E \cap \ttF_3 ) \ge 0 \,,\quad
  \tdim(E \cap \ttF_2 ) \ge 2 \,,\quad
  \tdim( E \cap \ttF_1) \ge 3 \,,\quad
  \tdim( E\cap \ttF_0) \ge 5 \, .
$$ 
\end{example}

\subsection*{Partitions versus $\sfa,\ttJ$}
It is well-known that Schubert varieties in $X=\tGr(\tti, n+1)$ are indexed by partitions 
\begin{equation} \label{E:Apart}
  \lambda \,=\, (\lambda_1,\ldots,\lambda_\tti) \in \bZ^\tti 
  \quad\hbox{such that}\quad
  1 \le \lambda_1 < \lambda_2 < \cdots < \lambda_\tti \le n+1 \,,
\end{equation}
cf. \cite[\S3.1.3]{MR1782635}.  Fix a flag $0 \subset F^1 \subset F^2 \subset\cdots\subset F^{n+1}$.  The corresponding Schubert variety is
\begin{equation}\label{E:Y1}
  Y_\lambda(F^\sbullet) \ := \ \{ E \in X \ | \ 
  \tdim( E \cap F^{\lambda_k}) \ge k \,,\ \forall \ k \} \,.
\end{equation}
Note that, if $\lambda_{k+1} = \lambda_k+1$, then the condition $\tdim(E\cap F^{\lambda_k}) \ge k$ is redundant; it is implied by $\tdim(E\cap F^{\lambda_{k+1}}) \ge k+1$.  To remove the redundancies, decompose $\lambda = \m_\sfp\cdots\m_1\m_0$ into maximal blocks of consecutive integers.  For example, if $\lambda = (2,3,4,7,8,12)$, then $\m_2 = (2,3,4)$, $\m_1 = (7,8)$ and $\m_0 = (12)$.  Let 
\begin{equation} \label{E:jA} 
  \ttj_\ell(\lambda) \ = \ |\m_\sfp\cdots\m_\ell|
\end{equation}
be the length of the sub-partition $\m_\sfp\cdots\m_\ell$.  (In all cases, $\ttj_0 = |\lambda| = \tti$.)  In the preceding example, $\ttj_2 = 3$, $\ttj_1 = 5$ and $\ttj_0=6$.  

\begin{remark} \label{R:tlambda}
Define $\tilde\lambda_k = \lambda_k-k$.  Then $0 \le \tilde\lambda_1 \le \tilde\lambda_2 \le \cdots \le \tilde\lambda_n\le n$.  Condense $\tilde\lambda$ by writing $\tilde\lambda = (\n_1^{c_1},\ldots,\n_t^{c_t})$ with $0 \le \n_1 < \n_2 < \cdots \n_t \le n$ and $0 < c_i$.  Then, $t = \sfp+1$, and $\ttj_\sfp = c_1$, $\ttj_{\sfp-1} = c_1 + c_2$, $\ttj_{\sfp-2} = c_1+c_2+c_3$, et cetera:  $\ttj_{\sfp+1-s} = c_1+\cdots+c_s$, and $\ttj_{\sfp-s} - \ttj_{\sfp+1-s} = c_{s+1}$.
\end{remark}

Note that $\lambda_{\ttj_\ell}$ is the last entry in the block $\m_\ell$.  That is,  
$$
  \{ \lambda_{\ttj_\sfp} , \ldots , \lambda_{{\ttj_1}} \} \ = \ 
  \{ \lambda_k \in \lambda \ | \ \lambda_k - \lambda_{k+1} > 1 \,,\ \ 1 \le k < \tti \} \,.
$$
The redundancy-free formulation of \eqref{E:Y1} is
\begin{equation} \label{E:Y2_app}
  Y_\lambda(F^\bullet) \ = \ \{ E \in X \ | \ 
  \tdim(E\cap F^{\lambda_{\ttj_\ell}}) \ge \ttj_\ell\,,\ 1 \le \ell \le \sfp \} \, .
\end{equation}
The following is \cite[Proposition 3.30]{MR2960030}.

\begin{lemma}[{\cite{MR2960030}}] \label{L:part}
Let $\lambda = (\lambda_1,\ldots,\lambda_\tti)$ be a partition satisfying \eqref{E:Apart}, and let $\lambda = \m_\sfp\cdots\m_1\m_0$ be the decomposition of $\lambda$ into maximal blocks of consecutive integers.  The pair $\sfa$, $\ttJ = \{\ttj_\sfp,\ldots,\ttj_1,\ttk_1,\ldots,\ttk_\sfq\}$ associated to the Schubert class $\xi_\lambda$ is given by \eqref{E:jA}, 
$$
  \{ \ttk_1 , \ldots , \ttk_\sfq \} \ = \ 
  \{ \tti - \ttj_\sfp + \lambda_{\ttj_\sfp} \,,\ldots ,\,
  \tti - \ttj_1+\lambda_{\ttj_1} \,,\, \lambda_\tti\} 
  \backslash\{\tti,n+1\} \,,
$$
and
$$
  \sfa \ = \ \left\{ \begin{array}{ll}
    \sfp  & \hbox{ if } \lambda_1 > 1\\
    \sfp-1  & \hbox{ if } \lambda_1 = 1 
  \end{array} \right\} \ = \ \left\{ \begin{array}{ll}
    \sfq \,, & \hbox{ if } \lambda_\tti = n+1 \\
    \sfq-1 \,, & \hbox{ if } \lambda_\tti < n+1 \,.
  \end{array} \right.
$$
Conversely, given $\sfa,\ttJ$, the associated partition $\lambda = \m_\sfp \cdots \m_1\m_0$ is given by 
$$
  \m_\ell \ = \ 
  ( \ttj_{\ell+1} + \ttk_m - \tti + 1 \,,\ldots,\, 
    \ttj_\ell + \ttk_m - \tti) \,,
$$
with $\ell+m = \sfa+1$.
\end{lemma}

\begin{example}
In Example \ref{eg:Gr0}, we have $\lambda = (3,4,7,11,12)$.
\end{example}

\begin{remark}
Together \cite[Proposition 3.30]{MR2960030} and \cite[Theorem 9.3.1]{MR1782635} imply that the integer $\sfa(w)$ is the number of irreducible components of $\tSing(X_w)$.  The relationship between $\sfa(w)$ and the number of irreducible components in $\tSing(X_w)$ for the other (Lagrangian Grassmannian, Spinor variety, Cayley plane and Freudenthal variety) irreducible, cominuscule $X$ is given by \cite[Corollary 3.6 and Table 1]{sing}.
\end{remark}

\subsection{Lagrangian Grassmannians $\tLG(n,2n) = C_n/P_n$} \label{S:aJC}

There exists a bijection between $W^\fp\backslash\{1,w_0\}$ and pairs $\sfa \ge0$ and $\ttJ = \{ \ttj_\sfp , \ldots , \ttj_1 \} \subset \{1,\ldots,n-1\}$ satisfying 
\begin{equation}\label{E:CD_J}
  1 \ \le \ \ttj_\sfp \,<\,\cdots \,<\, \ttj_1 \  \le \ n-1
\end{equation}
and 
$$
  \sfp \ \in \ \{\sfa, \sfa+1\}\,; 
$$
see  \cite[Corollary 3.17]{MR2960030}.  (These $\ttj_\ell$ have the opposite order of those in \cite{MR2960030}.)  For convenience we set
\begin{equation}\label{E:CDendpts}
  \ttj_{\sfp+1} \, := \, 0 \, , \quad \ttj_0 \, := \, n \, .
\end{equation}

To describe the Schubert variety $X_w$, fix a nondegenerate skew-symmetric bilinear form $( \cdot , \cdot )$ on $\bC^{2n}$, and basis $\{ e_1 , \ldots , e_{2n}\}$ of $\bC^{2n}$ satisfying $(e_a,e_b) = 0 = (e_{n+a},e_{n+b})$ and $(e_a,e_{n+b}) = \d_{ab}$ for all $1 \le a,b\le n$.  Then $X=\tLG(n,2n)$ is the $C_n$--orbit of $\langle e_1,\ldots,e_n\rangle$.  The abelian subalgebra $\fg_{-1}$ (which may be identified with $n$-by-$n$ symmetric matrices) is spanned by the root vectors $\{e^j_{n+k} + e^k_{n+j} \ | \ 1\le j \le k \le n\}$, with roots $-(\a_j+\cdots+\a_{k-1})-2(\a_k+\cdots+\a_n)$.  Define a filtration $\ttF_\sfp \subset \cdots \subset \ttF_{1} \subset \ttF_0$ of $\bC^{2n}$ by 
\begin{equation} \label{E:XwCa}
  \ttF_\ell \ = \ 
  \langle e_1 , \ldots , e_{\ttj_\ell} \,,\, 
          e_{n+\ttj_m+1},\ldots,e_{2n} \rangle 
\end{equation}
with 
\begin{equation} \label{E:XwCc}
  \ell+m \ =\  \sfa+1 \,.
\end{equation}
Set $o = [e_1\wedge\cdots\wedge e_n] \in \bP(\tw^n\bC^{2n})$.  Then
\begin{equation} \label{E:XwCb}
  X_w \ = \ \{ E\in X \ | \ 
  \tdim(E\cap \ttF_\ell) \ge \ttj_\ell \, , \ 
  \forall \ 0 \le \ell \le \sfp \} \,.
\end{equation}

\begin{example}\label{eg:LG0}
Consider $X = \tLG(5,10) \simeq C_5/P_5$.  The marking $\ttJ = \{ 1,2,4\}$ and integer $\sfa=3$ define the filtration $\ttF_3 \subset \ttF_2 \subset \ttF_1 \subset \ttF_0$ of $\bC^{10}$ as 
\begin{eqnarray*}
  \ttF_3 & = & \langle e_1 \,,\, e_{10} \rangle \,,\quad 
  \ttF_2 \ = \ \langle e_1 , e_2 \,,\, e_8 , e_9 , e_{10} \rangle \,,\\
  \ttF_1 & = & \langle e_1,\ldots,e_4 \,,\, e_7,\ldots,e_{10} \rangle \,,\quad
  \ttF_0 \ = \ \bC^{10} \, .
\end{eqnarray*}
The associated Schubert variety is the set of all $E \in \tLG(5,10)$ such that
$$
  \tdim(E\cap \ttF_3) \ge 1 \,,\quad
  \tdim(E\cap \ttF_2) \ge 2 \,,\quad
  \tdim(E\cap \ttF_1) \ge 4 \, .
$$
\end{example}

\subsection*{Partitions versus $\sfa,\ttJ$}

It is well-known that Schubert varieties in $X = \tLG(n,2n)$ are indexed by partitions $\lambda = (\lambda_1,\lambda_2,\ldots,\lambda_n)$ such that  
\begin{subequations}\label{SE:Cpart}
\begin{eqnarray}
  & & 1 \le \lambda_1 < \lambda_2 < \cdots < \lambda_n \le 2n \,, \quad \hbox{and} \\
\label{E:Cpart}
  & & \lambda_i \in \lambda \quad\hbox{if and only if} \quad 
  2n+1-\lambda_i\not\in\lambda \,, 
\end{eqnarray} 
\end{subequations}
cf. \cite[\S9.3]{MR1782635}.  Fix an isotropic flag $F^\sbullet$ in $\bC^{2n}$ by specifying $F^k = \langle e_1,\ldots,e_k\rangle$ and $(F^k,F^{2n-k}) = 0$ for $1\le k \le n$.  The corresponding Schubert variety is given by \eqref{E:Y1}.  As was the case in Section \ref{S:aJA}, this formulation is redundant.  Again, to remove the redundancies, decompose $\lambda = \m_\sfp\cdots\m_1\m_0$ into maximal blocks of consecutive integers.  Let $\ttj_\ell(\lambda) = |\m_\sfp\cdots\m_\ell|$ be the length of the sub-partition $\m_\sfp\cdots\m_s$.  (In all cases, $\ttj_0 = n$.)  Then the redundancy-free formulation is \eqref{E:Y2_app}.  A comparison of this with \eqref{E:XwCb} yields $\ttJ = \{ \ttj_\sfp(\lambda) , \ldots , \ttj_1(\lambda) \}$, and $\lambda_{\ttj_\ell} = \tdim\,\ttF_\ell = \ttj_\ell+n-\ttj_m$, for $\ell+m=\sfa+1$; note that $\lambda_{\ttj_\ell} + \lambda_{\ttj_m} = 2n$.  In order to determine the value of $\sfa$, we consider two cases:  
\begin{circlist}
\item First, suppose that $\lambda_1>1$.  By \eqref{E:Cpart}, this is equivalent to $\lambda_n=\lambda_{\ttj_0}=2n$.  Additionally, \eqref{E:Cpart} yields $\lambda_{\ttj_\ell}+\lambda_{\ttj_m}=2n$ when $\ell+m=\sfp+1$.  Therefore, $\sfa=\sfp$. 
\item Second, suppose that $\lambda_1=1$, equivalently, $\lambda_n=\lambda_{\ttj_0}<2n$.  Equation \eqref{E:Cpart} implies $\lambda_{\ttj_\ell} + \lambda_{\ttj_m} = 2n$ when $\ell+m = \sfp$.  Thus, $\sfp=\sfa+1$.
\end{circlist}
As an example, Table \ref{t:C5P5} lists the partitions $\lambda$ and corresponding $\sfa:\ttJ$ values for the Schubert varieties in $\tLG(5,10)$.
\begin{small}
\begin{table}
\caption[Lagrangian Grassmannian $\tLG(5,10)$]{Schubert varieties of $\tLG(5,10)$, Schur rigid classes indicated by $\ast$.}
\label{t:C5P5}
\renewcommand{\arraystretch}{1.2}
\begin{tabular}{|c|c||c|c||c|c|}
\hline
   $\lambda$ & $\sfa:\ttJ$ & $\lambda$ & $\sfa:\ttJ$ 
   & $\lambda$ & $\sfa:\ttJ$ \\ \hline \hline
      ${}^\mathbf{\ast}(1,2,3,4,5)$ &         & 
      $(1,2,3,4,6)$ & $0:4$ & 
      $(1,2,3,5,7)$ & $1:3,4$ \\ \hline
      $(1,2,4,5,8)$ & $1:2,4$ & 
      ${}^\mathbf{\ast}(1,2,3,6,7)$ & $0:3$ & 
      $(1,3,4,5,9)$ & $1:1,4$ \\ \hline
      $(1,2,4,6,8)$ & $2:2,3,4$ & 
      ${}^\mathbf{\ast}(2,3,4,5,10)$ & $1:4$ & 
      $(1,3,4,6,9)$ & $2:1,3,4$ \\ \hline
      $(1,2,5,7,8)$ & $1:2,3$ & 
      $(2,3,4,6,10)$ & $2:3,4$ & 
      $(1,3,5,7,9)$ & $3:1,2,3,4$ \\ \hline
      ${}^\mathbf{\ast}(1,2,6,7,8)$ & $0:2$ & 
      $(2,3,5,7,10)$ & $3:2,3,4$ & 
      ${}^\mathbf{\ast}(1,4,5,8,9)$ & $1:1,3$ \\ \hline
      $(1,3,6,7,9)$ & $2:1,2,4$ & 
      $(2,4,5,8,10)$ & $3:1,3,4$ & 
      ${}^\mathbf{\ast}(2,3,6,7,10)$ & $2:2,4$ \\ \hline
      $(1,4,6,8,9)$ & $2:1,2,3$ & 
      ${}^\mathbf{\ast}(3,4,5,9,10)$ & $1:3$ & 
      $(2,4,6,8,10)$ & $4:1,2,3,4$ \\ \hline
      $(1,5,7,8,9)$ & $1:1,2$ & 
      $(3,4,6,9,10)$ & $2:2,3$ & 
      $(2,5,7,8,10)$ & $3:1,2,4$ \\ \hline
      ${}^\mathbf{\ast}(1,6,7,8,9)$ & $0:1$ & 
      $(3,5,7,9,10)$ & $3:1,2,3$ & 
      $(2,6,7,8,10)$ & $2:1,4$ \\ \hline
      ${}^\mathbf{\ast}(4,5,8,9,10)$ & $1:2$ & 
      $(3,6,7,9,10)$ & $2:1,3$ & 
      $(4,6,8,9,10)$ & $2:1,2$ \\ \hline
      $(5,7,8,9,10)$ & $1:1$ & 
      ${}^\mathbf{\ast}(6,7,8,9,10)$ &       &  &  \\
\hline
\end{tabular}
\end{table}
\end{small}
The preceding discussion may be summarized as follows.

\begin{lemma} \label{L:Cpart}
Let $\lambda = (\lambda_1,\ldots,\lambda_n)$ be a partition satisfying \eqref{SE:Cpart}.  Let $\lambda = \m_\sfp\cdots\m_1\m_0$ be a decomposition of $\lambda$ into $\sfp+1$ maximal blocks of consecutive integers.  Then $\ttJ(\lambda) = \{ \ttj_\sfp(\lambda) , \ldots , \ttj_1(\lambda)\}$ is given by \eqref{E:jA}, and 
$$
  \sfa(\lambda) \ = \ \left\{ \begin{array}{ll}
  \sfp - 1 & \hbox{ if } \lambda_1 = 1 \\
  \sfp & \hbox{ if } \lambda_1 > 1 \, .
  \end{array}\right.
$$
Conversely, given $\sfa$ and $\ttJ=\{\ttj_\sfp,\cdots,\ttj_1\}$ we construct $\lambda(\sfa,\ttJ)=\m_\sfp(\sfa,\ttJ)\cdots\m_0(\sfa,\ttJ)$ by
\begin{equation} \label{E:Cblock}
  \m_\ell(\sfa,\ttJ) \ = \
  (n+1+\ttj_{\ell+1}-\ttj_m \,,\, \ldots \,,\, n+\ttj_\ell-\ttj_m ) \,,
\end{equation}
with $\ell+m=\sfa+1$.
\end{lemma}

\subsection{Spinor varieties $\cS_n = D_n/P_n$} \label{S:aJD}

Given $\sfa = \sfa(w)$ and $\ttJ = \ttJ(w)$, note that 
$$
  \a_{n-1}(Z_w) = 0 \ \hbox{ if } n-1\not\in\ttJ \,,\quad
  \hbox{ and } \quad  \a_{n-1}(Z_w) = 1 \ \hbox{ if } n-1\in\ttJ \,.
$$
Define 
\begin{equation} \label{E:r}
  \sfr \ = \ \left\lceil \half \left(\sfa + \a_{n-1}(Z_w)\right) \right\rceil \ = \ 
  \left\{ \renewcommand{\arraystretch}{1.2} \begin{array}{ll}
    \lceil \sfa/2 \rceil & \hbox{ if } n-1\not\in\ttJ \,,\\
    \lfloor \sfa/2 \rfloor + 1 & \hbox{ if } n-1\in\ttJ \,;
  \end{array} \right.
\end{equation}
see also \eqref{E:r2}.  There exists a bijection between $W^\fp\backslash\{1,w_0\}$, and pairs $\sfa \ge0$ and $\ttJ = \{ \ttj_\sfp , \ldots , \ttj_1 \} \subset \{1,\ldots,n-1\}$, ordered by \eqref{E:CD_J} and satisfying
\begin{equation} \label{E:Dr}
  \sfp - \a_{n-1}(Z_w) \,\in\, \{ \sfa , \sfa+1 \}\,, \quad\hbox{and}\quad
  2 \le \ttj_\sfr - \ttj_{\sfr+1} \ \hbox{ when } \ \sfr > \a_{n-1}(Z_w) \, ;
\end{equation}
see  \cite[Corollary 3.17]{MR2960030}.  (These $\ttj_\ell$ have the opposite order of those in \cite{MR2960030}.)   We maintain the convention \eqref{E:CDendpts}.

To describe the Schubert variety $X_w$, fix a nondegenerate symmetric bilinear form $(\cdot,\cdot)$ on $\bC^{2n}$ and basis $\{ e_1 , \ldots , e_{2n}\}$ of $\bC^{2n}$ satisfying $(e_a,e_b) = 0 = (e_{n+a},e_{n+b})$ and $(e_a,e_{n+b}) = \d_{ab}$ for all $1 \le a,b\le n$.   Our convention is that $\cS_n$ is the $D_n$--orbit of $\langle e_1,\ldots,e_n\rangle$.  The abelian subalgebra $\fg_{-1}$ (which may be identified with $n$-by-$n$ skew-symmetric matrices) is spanned by root vectors $\{ e^j_{n+k} - e^k_{n+j} \ | \ 1 \le j < k \le n \}$.  The corresponding roots are $-(\a_j+\cdots+\a_{n-2})-\a_n$, if $k=n$; $-(\a_j + \cdots + \a_n)$, if $k=n-1$; and $-(\a_j+\cdots+\a_{k-1})-2(\a_k+\cdots+\a_{n-2})-\a_{n-1}-\a_n$, if $k<n-1$.  Define a filtration $\ttF_\sfp \subset \cdots \subset \ttF_{1} \subset \ttF_0$ of $\bC^{2n}$ by \eqref{E:XwCa} with
\begin{equation}
  \ell+m \ = \
  \left\{ \begin{array}{l} \sfa+1 \quad\hbox{if } n-1\not\in\ttJ \,,\\
           \sfa+2 \quad\hbox{if } n-1\in\ttJ \end{array}\right\}
  \ = \ \sfa+1 + \a_{n-1}(Z_w) \, .
\end{equation}
Set $o = [e_1\wedge\cdots\wedge e_n] \in \bP(\tw^n\bC^{2n})$.  Then the Schubert variety $X_w$ is given by \eqref{E:XwCb}.

\subsection*{Partitions versus $\sfa,\ttJ$}

It is well-known that the Schubert varieties of $X=\cS_n$ are indexed by partitions $\lambda = (\lambda_1,\lambda_2,\ldots,\lambda_n)$ satisfying \eqref{SE:Cpart} and 
\begin{equation} \label{E:Dpart}
  \# \{ i \ | \ \lambda_i > n \} \quad\hbox{is even,}
\end{equation}  
cf. \cite[\S9.3]{MR1782635}.  Fix an isotropic flag $F^\sbullet$ in $\bC^{2n}$ by specifying $F^k = \langle e_1,\ldots,e_k\rangle$ and $(F^k,F^{2n-k}) = 0$ for $1\le k \le n$.  The corresponding Schubert variety is given by \eqref{E:Y1}.

We define $\ttJ(\lambda)$ as in \eqref{E:jA}, with the following modification of the block decomposition.  In the block decomposition $\lambda = \hat\m_p\cdots\hat\m_1\hat\m_0$, the integers $n-1,n+1$ are considered `consecutive' and are placed in the same $\hat\m_s$--block; likewise, the integers $n,n+2$ are `consecutive.'  For example, if $n=5$, then $\lambda = (2,3,4,6,10)$ has block decomposition $\hat\m_1\hat\m_0 = (2,3,4,6)(10)$; likewise, $\lambda = (1,2,5,7,8)$ has block decomposition $\hat\m_1\hat\m_0 = (1,2)(5,7,8)$.

As before,
\begin{equation} \label{E:JDn}
  \ttj_\ell(\lambda) \ = \ | \hat\m_\sfp\cdots\hat\m_\ell |\,.
\end{equation}
Define
\begin{equation}\label{E:aDn}
  \sfa \ = \ \left\{ \begin{array}{ll}
  \sfp - 2 & 
  \hbox{ if } \lambda_1 = 1 \hbox{ and } \lambda_n - \lambda_{n-1} > 1 \,,\\
  \sfp - 1 & 
  \hbox{ if } \lambda_1 = 1 \hbox{ and } \lambda_n - \lambda_{n-1} = 1 \,,
  \ \hbox{or } \lambda_1 > 1 \hbox{ and } \lambda_n - \lambda_{n-1} > 1 \,,\\
  \sfp & 
  \hbox{ if } \lambda_1 > 1 \hbox{ and }\lambda_n - \lambda_{n-1} = 1  \,.
  \end{array}\right.
\end{equation}
As an example, Table \ref{t:D6P6} lists the partitions and corresponding $\sfa:\ttJ$ and $\sfr$ values for the Schubert varieties of $\cS_6 = \tSpin_{12}\bC/P_6$.  The discussion above yields the following.

\begin{lemma} \label{L:Dpart}
Given a partition $\lambda$ indexing a Schubert variety \eqref{E:Y1} in $\cS_n = D_n/P_n$, the set $\ttJ(\lambda) = \{ \ttj_\sfp(\lambda) , \ldots , \ttj_1(\lambda)\}$ is given by \eqref{E:JDn}, and $\sfa(\lambda)$ is given by \eqref{E:aDn}.  

Conversely, given $\sfa$ and $\ttJ=\{\ttj_\sfp,\cdots,\ttj_1\}$, we construct $\lambda(\sfa,\ttJ)$ as follows.  Let $\lambda' = \m_\sfp\cdots\m_1\m_0$ be given by
\eqref{E:Cblock}, with $\ell+m=\sfa+1+\a_{n-1}(Z_w)$.  If $\lambda'$ satisfies \eqref{E:Dpart}, then $\lambda = \lambda'$.  If \eqref{E:Dpart} fails for $\lambda'$, then we modify the partition as follows: precisely one of $\{ n,n+1\}$ belongs to $\lambda'$, denote this element by $a'$, and the other by $a$.  Then $\lambda$ is obtained from $\lambda'$ by replacing $a'$ with $a$.  
\end{lemma}

\begin{small}
\begin{table}
\caption[Spinor variety $\cS_6$]{Schubert varieties of $\cS_6$, Schur rigid classes indicated by $\ast$.}
\label{t:D6P6}
\renewcommand{\arraystretch}{1.2}
\begin{tabular}{|c|c|c||c|c|c|}
\hline
   $\lambda$ & $\sfa:\ttJ$ & $\sfr$ & 
   $\lambda$ & $\sfa:\ttJ$ & $\sfr$ \\ \hline \hline
      ${}^\mathbf{\ast}(1,2,3,4,5,6)$ & &  &
      $(1,2,3,4,7,8)$ & $0:4$ & $0$ \\ \hline
      $(1,2,3,5,7,9)$ & $0:3,5$ & $1$ & 
      $(1,2,4,5,7,10)$ & $0:2,5$ & $1$ \\ \hline
      ${}^\mathbf{\ast}(1,2,3,6,8,9)$ & $0:3$ & $0$ & 
      $(1,3,4,5,7,11)$ & $0:1,5$ & $1$ \\ \hline
      $(1,2,4,6,8,10)$ & $1:2,3,5$ & $1$ & 
      ${}^\mathbf{\ast}(2,3,4,5,7,12)$ & $0:5$ & $1$ \\ \hline
      $(1,3,4,6,8,11)$ & $1:1,3,5$ & $1$ & 
      $(1,2,5,6,9,10)$ & $1:2,4$ & $1$ \\ \hline
      $(2,3,4,6,8,12)$ & $1:3,5$ & $1$ & 
      $(1,3,5,6,9,11)$ & $2:1,2,4,5$ & $2$ \\ \hline
      ${}^\mathbf{\ast}(1,2,7,8,9,10)$ & $0:2$ & $0$ & 
      $(2,3,5,6,9,12)$ & $2:2,4,5$ & $2$ \\ \hline
      ${}^\mathbf{\ast}(1,4,5,6,10,11)$ & $1:1,4$ & $1$ & 
      $(1,3,7,8,9,11)$ & $1:1,2,5$ & $1$ \\ \hline
      $(2,4,5,6,10,12)$ & $2:1,4,5$ & $2$ & 
      ${}^\mathbf{\ast}(2,3,7,8,9,12)$ & $1:2,5$ & $1$ \\ \hline
      $(1,4,7,8,10,11)$ & $2:1,2,4$ & $1$ & 
      ${}^\mathbf{\ast}(3,4,5,6,11,12)$ & $1:4$ & $1$ \\ \hline
      $(2,4,7,8,10,12)$ & $3:1,2,4,5$ & $2$ & 
      $(1,5,7,9,10,11)$ & $1:1,3$ & $1$ \\ \hline
      $(3,4,7,8,11,12)$ & $2:2,4$ & $1$ & 
      $(2,5,7,9,10,12)$ & $2:1,3,5$ & $2$ \\ \hline
      ${}^\mathbf{\ast}(1,6,8,9,10,11)$ & $0:1$ & $0$ & 
      $(3,5,7,9,11,12)$ & $3:1,3,4$ & $2$ \\ \hline
      $(2,6,8,9,10,12)$ & $1:1,5$ & $1$ & 
      ${}^\mathbf{\ast}(4,5,7,10,11,12)$ & $1:3$ & $1$ \\ \hline
      $(3,6,8,9,11,12)$ & $2:1,4$ & $1$ & 
      $(4,6,8,10,11,12)$ & $2:1,3$ & $1$ \\ \hline
      $(5,6,9,10,11,12)$ & $1:2$ & $1$ & 
      ${}^\mathbf{\ast}(7,8,9,10,11,12)$ & & \\ 
\hline
\end{tabular}
\end{table}
\end{small}

Since $\ttj_1 = n-1$ (equivalently, $\a_{n-1}(Z_w) = 1$) if and only if $\lambda_n > \lambda_{n-1} + 1$, \eqref{E:aDn} is equivalent to
\begin{equation} \nonumber
  \sfa(\lambda) + \a_{n-1}(Z_w) \ = \ \left\{ \begin{array}{ll}
  \sfp - 1 & \hbox{ if } \lambda_1 = 1 \,,\\
  \sfp  & \hbox{ if } \lambda_1 > 1 \,.
  \end{array}\right.
\end{equation}
Similarly, \eqref{E:r} may be expressed as
\begin{equation} \label{E:r2}  
  \sfr \ = \ 
  \left\{ \renewcommand{\arraystretch}{1.2} \begin{array}{ll}
    \lfloor \sfp/2 \rfloor & \hbox{ if } \lambda_1 = 1 \,,\\
    \lceil \sfp/2 \rceil & \hbox{ if } \lambda_1 > 1 \,.
  \end{array} \right.
\end{equation}

\bibliographystyle{plain}

\begin{thebibliography}{10}

\bibitem{MR1782635}
Sara Billey and V.~Lakshmibai.
\newblock {\em Singular loci of {S}chubert varieties}, volume 182 of {\em
  Progress in Mathematics}.
\newblock Birkh\"auser Boston Inc., Boston, MA, 2000.

\bibitem{MR0149503}
Armand Borel and Andr{\'e} Haefliger.
\newblock La classe d'homologie fondamentale d'un espace analytique.
\newblock {\em Bull. Soc. Math. France}, 89:461--513, 1961.

\bibitem{SchurRigid}
R.~L. Bryant.
\newblock {\em Rigidity and Quasi-Rigidity of Extremal Cycles in {H}ermitian
  Symmetric Spaces}, volume 153 of {\em Annals of Mathematics Studies}.
\newblock Princeton University Press, 2010.
\newblock arXiv:math/0006186.

\bibitem{MR916718}
Robert~L. Bryant.
\newblock Metrics with exceptional holonomy.
\newblock {\em Ann. of Math. (2)}, 126(3):525--576, 1987.

\bibitem{MR2532439}
Andreas {\v{C}}ap and Jan Slov{\'a}k.
\newblock {\em Parabolic geometries. {I}}, volume 154 of {\em Mathematical
  Surveys and Monographs}.
\newblock American Mathematical Society, Providence, RI, 2009.
\newblock Background and general theory.

\bibitem{coskunOrth}
Izzet Coskun.
\newblock Rigidity of {S}chubert classes in orthogonal {G}rassmannians.
\newblock {\em Israel J. Math.}
\newblock to appear.

\bibitem{MR2819546}
Izzet Coskun.
\newblock Rigid and non-smoothable {S}chubert classes.
\newblock {\em J. Differential Geom.}, 87(3):493--514, 2011.

\bibitem{CR1}
Izzet Coskun and Colleen Robles.
\newblock Flexibility of {S}chubert classes.
\newblock submitted, arXiv:1303.0253, 2013.

\bibitem{MR0357402}
Robin Hartshorne, Elmer Rees, and Emery Thomas.
\newblock Nonsmoothing of algebraic cycles on {G}rassmann varieties.
\newblock {\em Bull. Amer. Math. Soc.}, 80:847--851, 1974.

\bibitem{MR0224611}
Heisuke Hironaka.
\newblock Smoothing of algebraic cycles of small dimensions.
\newblock {\em Amer. J. Math.}, 90:1--54, 1968.

\bibitem{MR2191767}
Jaehyun Hong.
\newblock Rigidity of singular {S}chubert varieties in {${\rm Gr}(m,n)$}.
\newblock {\em J. Differential Geom.}, 71(1):1--22, 2005.

\bibitem{MR2276624}
Jaehyun Hong.
\newblock Rigidity of smooth {S}chubert varieties in {H}ermitian symmetric
  spaces.
\newblock {\em Trans. Amer. Math. Soc.}, 359(5):2361--2381 (electronic), 2007.

\bibitem{MR0265371}
Steven~L. Kleiman.
\newblock Geometry on {G}rassmannians and applications to splitting bundles and
  smoothing cycles.
\newblock {\em Inst. Hautes \'Etudes Sci. Publ. Math.}, (36):281--297, 1969.

\bibitem{MR0285535}
Steven~L. Kleiman and John Landolfi.
\newblock Singularities of special {S}chubert varieties.
\newblock In {\em Symposia {M}athematica, {V}ol. {V} ({INDAM}, {R}ome,
  1969/70)}, pages 341--346. Academic Press, London, 1969/1970.

\bibitem{MR1920389}
Anthony~W. Knapp.
\newblock {\em Lie groups beyond an introduction}, volume 140 of {\em Progress
  in Mathematics}.
\newblock Birkh\"auser Boston Inc., Boston, MA, second edition, 2002.

\bibitem{MR0142696}
Bertram Kostant.
\newblock Lie algebra cohomology and the generalized {B}orel-{W}eil theorem.
\newblock {\em Ann. of Math. (2)}, 74:329--387, 1961.

\bibitem{MR0142697}
Bertram Kostant.
\newblock Lie algebra cohomology and generalized {S}chubert cells.
\newblock {\em Ann. of Math. (2)}, 77:72--144, 1963.

\bibitem{sing}
C.~Robles.
\newblock Singular loci of cominuscule {S}chubert varieties.
\newblock arXiv:1203.0324, 2012.

\bibitem{MR2960030}
C.~Robles and D.~The.
\newblock Rigid {S}chubert varieties in compact {H}ermitian symmetric spaces.
\newblock {\em Selecta Math. (N.S.)}, 18(3):717--777, 2012.

\bibitem{MR2538022}
Hugh Thomas and Alexander Yong.
\newblock A combinatorial rule for (co)minuscule {S}chubert calculus.
\newblock {\em Adv. Math.}, 222(2):596--620, 2009.

\bibitem{Walters}
Maria Walters.
\newblock {\em Geometry and uniqueness of some extreme subvarieties in complex
  Grassmannians}.
\newblock PhD thesis, University of Michigan, 1997.

\end{thebibliography}
\def\cprime{$'$}

\end{document}